\documentclass{amsart}
\usepackage{amsfonts,amssymb,amscd,amsmath,enumerate,verbatim,calc}
\usepackage[all]{xy}

\newcommand{\CM}{Cohen-Macaulay}

\newcommand{\ZZ}{\mathbb{Z} }

\newcommand{\wrt}{with respect to}
\newcommand{\D}{\mathcal{D} }

\newcommand{\bY}{\mathbf{Y}}
\newcommand{\bX}{\mathbf{X}}

\newcommand{\n}{\mathfrak{n} }
\newcommand{\m}{\mathfrak{m} }
\newcommand{\Q}{\mathfrak{Q} }
\newcommand{\M}{\mathbf{\mathcal{M}} }
\newcommand{\N}{\mathbf{\mathcal{N}} }
\newcommand{\E}{\mathbf{\mathcal{E}} }

\newcommand{\LL}{\mathbf{\mathcal{L}} }
\newcommand{\FF}{\mathcal{T}}

\newcommand{\R}{\mathcal{R} }
\newcommand{\Z}{\mathbb{Z} }

\newcommand{\rt}{\rightarrow}

\newcommand{\ov}{\overline}

\newcommand{\wh}{\widehat }

\newcommand{\Ass}{\operatorname{Ass}}

\newcommand{\depth}{\operatorname{depth}}
\newcommand{\fg}{\operatorname{fg}}

\newcommand{\height}{\operatorname{height}}
\newcommand{\injdim}{\operatorname{injdim}}
\newcommand{\Proj}{\operatorname{Proj}}
\newcommand{\Spec}{\operatorname{Spec}}
\newcommand{\Supp}{\operatorname{Supp}}

\newcommand{\Hom}{\operatorname{Hom}}
\newcommand{\Ext}{\operatorname{Ext}}

\theoremstyle{plain}

\newtheorem{theorem}{Theorem}[section]
\newtheorem{corollary}[theorem]{Corollary}
\newtheorem{lemma}[theorem]{Lemma}
\newtheorem{proposition}[theorem]{Proposition}
\newtheorem{question}[theorem]{Question}

\theoremstyle{definition}
\newtheorem{definition}[theorem]{Definition}
\newtheorem{s}[theorem]{}
\newtheorem{remark}[theorem]{Remark}
\newtheorem{example}[theorem]{Example}

\newtheorem{construction}[theorem]{Construction}
\theoremstyle{remark}

\begin{document}

\title{Koszul homology of $F$-finite module and applications}
\author{Tony~J.~Puthenpurakal}
\date{\today}
\address{Department of Mathematics, IIT Bombay, Powai, Mumbai 400 076}

\email{tputhen@math.iitb.ac.in}

 \begin{abstract}
Let $k$ be an infinite field of characteristic $p > 0$ and let $R = k[Y_1,\ldots, Y_d]$ (or $R = k[[Y_1,\ldots, Y_d]]$). Let $F \colon \text{Mod}(R) \rt \text{Mod}(R)$ be the Frobenius functor and let $\M$ be a $F_R$-finite module (in the sense of Lyubeznik \cite{Lyu-2}). We show that if $r \geq 1$ then the Koszul homology modules \\
 $H_i(Y_1,\ldots, Y_r; \M)$ are $F_{\ov{R}}$-finite  modules where $\overline{R} = R/(Y_1,\ldots, Y_r)$ for $i = 0, \ldots, r$. As an application if  $A$ is a regular ring containing a field of characteristic $p > 0$ and $S = A[X_1,\ldots, X_m]$ is standard graded and $I$ is an arbitrary graded ideal in $S$ then we  give a comprehensive study of graded components of local cohomology  modules $H^i_I(S)$.
 This extends in positive characteristic results we proved in \cite{P}. We study $H^i_I(S) $ when $A$ is local and prove that if $S/I$ is equidimensional and $\Proj(S/I)$ is \CM \ then $H^i_I(S)_n = 0$ for all $n \geq 0$ and for all $i > \height I$. If $(B,\n)$ is a equicharacteristic local Noetherian ring with infinite residue field and  with a surjective map $\pi \colon T \rt B$ where $(T,\n)$ is regular local then we show that the Koszul cohomology modules $H^j(\n, H^{\dim T - i}_{\ker \pi }(T))$ depend only on $A, i, j$ and not on $T$ and $\pi$.
\end{abstract}
 \maketitle
\section{introduction}
Let $k_0$ be a field of characteristic zero and let   $R_0 = k_0[Y_1, \ldots, Y_d]$. Let $D_{ k_0}(R_0) = k_0<Y_1, \ldots, Y_d, \partial_1, \ldots, \partial_d>$ (where $\partial_i = \partial_i/\partial Y_i$) be  the ring of $k_0$-differential operators on $R_0$.  Let $\N$ be a holonomic  $D_{ k_0}(R_0)$-module. Fix $r$ with $1\leq r \leq d$. Then the De Rham cohomology modules $H^i(\partial_1, \ldots, \partial_r; \N)$ and the Koszul cohomology modules $H^i(Y_1,\ldots, Y_r; \N)$  are holonomic $D_{ k_0}(R_0/(Y_1, \ldots, Y_r))$-modules; (see \cite[6.3]{Bjork}
for de Rham cohomology, the case of Koszul cohomology is similar). One might ask whether a similar result is true when $k_0$ is replaced by a field $k$ of characteristic $p > 0$. In this case for De Rham cohomology  it is well-known that the analogous result is false, for instance see \ref{bad-de-Rham}. We do not know where the analogous result is true for Koszul cohomology modules. The purpose of this paper is to prove that Koszul cohomology  behaves well of a large subclass of holonomic $D_k(k[Y_1,\ldots, Y_m])$-modules.

For the rest of this paper $R$ will denote a regular ring containing a field of $k$ of characteristic $p > 0$. Let $ F \colon \text{Mod}(R) \rt \text{Mod}(R) $ be the Frobenius functor  of Peskine-Szpiro \cite{PS}. For the definition of a $F_R$-module see \cite{Lyu-2}. Let $F_R$-Mod be the category of $F_R$-modules and let $F_R$-mod be the category of $F_R$- finite modules.

Our result regarding the Koszul cohomology is

\begin{theorem}\label{main-koszul}
Let $k$ be an infinite field of characteristic $p > 0$. Let $R$ be one of the following regular rings
\begin{enumerate}[\rm (i)]
  \item $k[Y_1, \ldots, Y_d]$.
  \item $k[[Y_1,\ldots, Y_d]]$.
  \item $A[X_1, \ldots, X_m]$ where $A = k[[Y_1,\ldots, Y_d]]$.
\end{enumerate}
Let $\M$ be a $F_R$-finite module. Fix $r \geq 1$. Then the Koszul homology   \\ modules $H_i(Y_1,\ldots, Y_r; \M)$ are
$F_{\ov{R}}$-finite  modules (where $\ov{R} = R/(Y_1,\ldots, Y_r)$) for $i = 0, \ldots,r$.
\end{theorem}

As a consequence we get
\begin{corollary}\label{cor-koszul}
Let $k$ be an infinite field of characteristic $p > 0$. Let $R$ be one of the following regular rings
\begin{enumerate}[\rm (i)]
  \item $k[Y_1, \ldots, Y_d]$.
  \item $k[[Y_1,\ldots, Y_d]]$.
\end{enumerate}
Let $\M$ be a $F_R$-finite module. Then the Koszul homology  modules \\ $H_i(Y_1,\ldots, Y_d; \M)$ are
finite dimensional $k$-vector spaces for $i = 0, \ldots,d$.
\end{corollary}

When $R$ is graded polynomial ring and the corresponding $F_R$-module is graded we can say more:
\begin{corollary}\label{cor-graded}
Let $k$ be an infinite field of characteristic $p > 0$. Let  \\ $R = k[Y_1, \ldots, Y_d]$ be standard graded.
Let $\M$ be a graded $F_R$-finite module. Then the Koszul homology  modules  $H_i(Y_1,\ldots, Y_d; \M)$ for $i = 0, \ldots, d$ are
concentrated in degree zero, i.e., $H_i(Y_1,\ldots, Y_d; \M)_r = 0$ for $r \neq 0$.
\end{corollary}

We now give application of Theorem \ref{main-koszul}
\s\label{std}  \textbf{Application-I:} \\
 \emph{Standard Assumption:} From henceforth $A$ will denote a regular ring containing a field of characteristic $p > 0$. Let $R = A[X_1,\ldots, X_m]$ be standard graded with $\deg A = 0$ and $\deg X_i = 1$ for all $i$. We also assume $m \geq 1$. Let $\FF$ be a
graded Lyubeznik functor on
$\ ^*Mod(R)$. Set $\M = \FF(R)$. Note $\M$ is a graded $R$-module. Set $\M = \bigoplus_{n \in \ZZ}\M_n$.

We extend in char $p > 0$ the following results which was proved in char $0$ in \cite{P}.

 \textbf{I:} \textit{(Vanishing:)}  The first result we prove is that vanishing of almost all graded components of $M$ implies vanishing of $M$. More precisely we show
 \begin{theorem}\label{vanish}
(with hypotheses as in \ref{std}).  If $\M _n = 0$ for all $|n|  \gg 0$ then
$\M = 0$.
 \end{theorem}

\textbf{II} (\textit{Tameness and Rigidity :}) In view of Theorem \ref{vanish}, it follows that if
$\M = \FF(R) = \bigoplus_{n \in \ZZ}\M_n $ is \emph{non-zero} then either $\M_n \neq 0$ for infinitely  many $n \ll 0$, OR, $\M_n \neq 0$ for infinitely  many $n \gg 0$. We show that $\M$ is \emph{tame} and \emph{rigid}. More precisely
\begin{theorem}\label{tame}
(with hypotheses as in \ref{std}).  Then we have
\begin{enumerate}[\rm (a)]
\item
The following assertions are equivalent:
\begin{enumerate}[\rm(i)]
\item
$\M_n \neq 0$ for infinitely many $n \ll 0$.
\item
There exists $r$ such that $\M_n \neq 0$ for all $n \leq r$.
\item
$\M_n \neq 0$ for all $n \leq -m$.
\item
$\M_r \neq 0$ for  some $r \leq -m$.
\end{enumerate}
\item
The following assertions are equivalent:
\begin{enumerate}[\rm(i)]
\item
$\M_n \neq 0$ for infinitely many $n \gg 0$.
\item
There exists $r$ such that $\M_n \neq 0$ for all $n \geq r$.
\item
$\M_n \neq 0$ for all $n \geq 0$.
\item
$\M_n \neq 0$ for some $n \geq 0$.
\end{enumerate}
\item
(When $m \geq 2$.)
The following assertions are equivalent:
\begin{enumerate}[\rm(i)]
\item
$\M_t \neq 0$ for some $t$ with $-m < t < 0$.
\item
$\M_n \neq 0$ for all $n \in \ZZ$.
\end{enumerate}
\end{enumerate}
\end{theorem}
We complement Theorem \ref{tame} by showing the following
\begin{example}\label{ex-tame} There exists a regular ring $A$ and homogeneous ideals $I, J, K, L$ in $R = A[X_1,\ldots, X_m]$ such that
\begin{enumerate}[\rm (i)]
\item
$H^i_I(R)_n \neq 0$ for all $n \leq -m$ and $H^i_I(R)_n = 0$ for all $n > - m$.
\item
$H^i_J(R)_n \neq 0$ for all $n \geq 0$ and $H^i_J(R)_n = 0$ for all $n < 0$.
\item
$H^i_K(R)_n \neq 0$ for all $n \in \ZZ$.
\item
$H^i_L(R)_n \neq 0$ for all $n \leq -m$, $H^i_L(R)_n \neq 0$ for all $n \geq 0$ and $H^i_L(R) = 0$ for all $n$ with $-m < n < 0$.
\end{enumerate}
\end{example}

In fact the examples in \cite[section 7]{P} is characteristic free (we needlessly assumed $A$ contains a field of characteristic zero).

\textbf{III} \textit{(Infinite generation:)} Recall that each component of $H^m_{R_+}(R)$ is a finitely generated $A$-module, cf., \cite[15.1.5]{BS}. We give a sufficient condition for infinite generation of a component of graded local cohomology module over $R$.
\begin{theorem}\label{inf-gen}(with hypotheses as in \ref{std}). Further assume $A$ is a domain. Let $I$ be a homogeneous ideal of $R$. Assume $I \cap A \neq 0$. If $H^i_I(R)_c \neq 0$ then
$H^i_I(R)_c$ is NOT finitely generated as an $A$-module.
\end{theorem}

\textbf{IV} (\textit{Bass numbers:}) The $j^{th}$ Bass number of an $A$-module $E$ with respect to a prime ideal $P$ is defined as $\mu_j(P,E) = \dim_{k(P)} \Ext^j_{A_P}(k(P), E_P)$ where $k(P)$ is the residue field of $A_P$. We note that if $E$ is finitely generated as an $A$-module  then $\mu_j(P,E)$ is a finite number (possibly zero) for all $j \geq 0$. In view of Theorem \ref{inf-gen} it is not clear whether $\mu_j(P, \FF(R))_n)$ is a finite number. Surprisingly  we have the following dichotomy:
\begin{theorem}
\label{bass-basic}(with hypotheses as in \ref{std}).  Let $P$ be a prime ideal in $A$. Fix $j \geq 0$. EXACTLY one of the following hold:
\begin{enumerate}[\rm(i)]
\item
$\mu_j(P, \M_n)$ is infinite for all $n \in \ZZ$.
\item
$\mu_j(P, \M_n)$ is finite for all $n \in \ZZ$. In this case EXACTLY one of the following holds:
\begin{enumerate}[\rm (a)]
\item
$\mu_j(P, \M_n) = 0$ for all $n \in \ZZ$.
\item
$\mu_j(P, \M_n) \neq 0$ for all $n \in \ZZ$.
\item
$\mu_j(P, \M_n) \neq 0$ for all $n  \geq 0$ and $\mu_j(P, \M_n) = 0$ for all
$n < 0$.
\item
$\mu_j(P, \M_n) \neq 0$ for all $n  \leq -m$ and $\mu_j(P, \M_n) = 0$ for all
$n > -m$.
\item
$\mu_j(P, \M_n) \neq 0$ for all $n  \leq -m$, $\mu_j(P, \M_n) \neq 0$ for all $n  \geq 0$  and $\mu_j(P, \M_n)  = 0$ for all $n$ with $-m < n < 0$.
\end{enumerate}
\end{enumerate}
\end{theorem}
There easy examples where (i) and (ii) hold, see \cite[section 7]{P}.  The only examples where the author was able to show (i) hold had $m \geq 2$.  Surprisingly  the following result holds.
\begin{theorem}
\label{bass-m-one}(with hypotheses as in \ref{std}).  Assume $m = 1$. Let $P$ be a prime ideal in $A$.  Then
$\mu_j(P, \M_n)$ is finite for all $n \in \ZZ$.
\end{theorem}

\textbf{V} (\textit{ Growth of Bass numbers}). Fix $j \geq 0$. Let $P$ be a prime ideal in $A$ such that $\mu_j(P, \FF(R)_n)$ is finite for all $n \in \ZZ$. We may ask about the growth of the function $n \mapsto \mu_j(P, \FF(R)_n)$ as $n \rt -\infty$ and when $n \rt + \infty$. We prove
\begin{theorem}
\label{bass-growth}(with hypotheses as in \ref{std}).  Let $P$ be a prime ideal in $A$. Let $j \geq 0$. Suppose $\mu_j(P, \M_n)$ is finite for all $n \in \ZZ$. Then there exists polynomials $f_\M^{j,P}(Z), g_\M^{j,P}(Z) \in \mathbb{Q}[Z]$ of degree $\leq m - 1$ such that
\[
f_\M^{j,P}(n) = \mu_j(P, \M_n) \ \text{for all} \ n \ll 0  \quad \text{AND} \quad  g_\M^{j,P}(n) = \mu_j(P, \M_n) \ \text{for all} \ n \gg 0.
\]
\end{theorem}

\textbf{VI}
(\textit{Associate primes:}) If $E = \bigoplus_{n \in \ZZ} E_n$ is a graded $R$-module then  there are two questions regarding asymptotic primes:

\textit{Question 1:}\textit{(Finiteness:)} Is the set
$\bigcup_{n \in \ZZ} \Ass_A E_n   $ finite?

\textit{Question 2:} \textit{(Stability:)} Does there exists integers $r , s$ such
that $\Ass_A E_n = \Ass_A E_r$ for all $n \leq r$ and $\Ass_A E_n = \Ass_A E_s$ for all $n \geq s$.

For graded local cohomology modules we show that both Questions above have affirmative answer for all  regular rings $A$. Note in characteristic zero we proved a more restrictive result. \cite[1.13]{P}.
\begin{theorem}\label{ass}(with hypotheses as in \ref{std}). Let $\M = \FF(R) = \bigoplus_{n \in \ZZ}M_n$. Then
\begin{enumerate}[\rm (1)]
\item
$\bigcup_{n \in \ZZ} \Ass_A \M_n   $ is a finite set.
\item
$\Ass_A \M_n = \Ass_A \M_m$ for all $n \leq -m$.
\item
$\Ass_A \M_n = \Ass_A \M_0$ for all $n \geq 0$.
\end{enumerate}
\end{theorem}

\textbf{VII}
(\textit{Dimension of Supports  and injective dimension:}) Let $E$ be an $A$-module. Let
$\injdim_A E$ denotes the injective dimension of $E$. Also
$\Supp_A E = \{ P \mid  E_P \neq 0 \ \text{and $P$ is a prime in $A$}\}$ is the support of an $A$-module $E$.
By $\dim_A E $ we mean the dimension of $\Supp_A E$ as a subspace of $\Spec(A)$.
We prove the following:
\begin{theorem}\label{injdim-and-dim}(with hypotheses as in \ref{std}). Let $\M = \FF(R) = \bigoplus_{n \in \ZZ}\M_n$. Then we have
\begin{enumerate}[\rm (1)]
\item
$\injdim \M_c \leq \dim \M_c$ for all $c \in \ZZ$.
\item
$\injdim \M_n = \injdim \M_{-m}$ for all $n \leq -m$.
\item
$\dim \M_n = \dim \M_{-m}$ for all $n \leq -m$.
\item
$\injdim \M_n = \injdim \M_{0}$ for all $n \geq 0$.
\item
$\dim \M_n = \dim \M_{0}$ for all $n \geq 0$.
 \item
 If $m \geq 2$ and $-m < r,s < 0$ then
 \begin{enumerate}[\rm (a)]
 \item
 $\injdim \M_r = \injdim \M_{s}$  and $\dim M_r = \dim M_{s}$.
 \item
 $\injdim \M_r \leq \min \{ \injdim \M_{-m}, \injdim \M_{0} \}$.
 \item
 $\dim \M_r \leq \min \{ \dim \M_{-m}, \dim \M_{0} \}$.
 \end{enumerate}
\end{enumerate}
\end{theorem}

\begin{remark}
It might be asked whether we can give characteristic free proofs of the above results on graded components. However note that in char 0, for the ring $k[X_1,\ldots, X_d]$ we also have
De Rham cohomology. This enables in characteristic zero more transparent proofs. As written in \cite{P} the main contribution in that paper was to guess the results.
In this paper we already knew what to prove. However the proofs are not as transparent as in characteristic zero.
\end{remark}

\s \label{fg} \textbf{Application-II:} We now give an application of our result for which we do not have an analogue in characteristic zero.
Let $A = k[[Y_1, \ldots, Y_d]]$ where $k$ is an infinite field of characteristic $p > 0$ and let $R = A[X_1, \ldots, X_m]$ be standard graded. Let $I$ be a homogeneous  ideal in $R$ and let $S = R/I = \bigoplus_{n \geq 0}S_n$.
Let $\n$ be the unique maximal homogeneous ideal of $S$. Set
\[
\fg_\n(S) = \max \{ r \mid H^i_{\n}(S)_n = 0 \ \text{for all but finitely many $n$ and all $i < r$} \}.
\]
Let $\Proj(S)$ consist of homogeneous prime ideals of $S$ which do not contain $S_+$. It can be shown \cite[2.2, 2.4]{HM}
\[
\fg_\n(S) = \min_{P \in \Proj{(S)}}\{ \depth S_P + \dim S/P \}.
\]
\begin{remark}\label{proj-remark}
If $S$ is equi-dimensional and if $\Proj(S)$ is non-empty and \CM \ then $\fg_\n(S) = \dim S$.
\end{remark}
We show
\begin{theorem}\label{app-fg}
(with hypotheses as in \ref{fg}). Let $i < \fg_\n(S)$. Then \\ $H^{\dim R -i}_I(R)_n = 0$ for all $n \geq 0$.
\end{theorem}
Recall Peskine and Szpiro proved that if $R/I$ is \CM  \ then $H^{j}_I(R) = 0$ for all $j > \height I$; see \cite[III.4.1]{PS}. Our  next result yields information when only $\Proj(S)$ is assumed to be
\CM.
\begin{corollary}
 Let $(A,\m)$ be a regular local ring and let $R = A[X_1,\ldots, X_m]$. Let $I$ be a graded ideal in $R$ such that $I$ is equidimensional and $\Proj (R/I)$ is \CM. Then $H^{j}_I(R)_n = 0$ for all $n \geq 0$ and for all $j > \height I$.
\end{corollary}

In the proof of Theorem \ref{app-fg} we use the face that $k$ is of characteristic $p > 0$. However the statement of Theorem \ref{app-fg} is not dependent on characteristic. So we ask
\begin{question}
  Does the assertion of Theorem \ref{app-fg} hold if $k$ is of characteristic zero?
\end{question}

In view of Theorem \ref{app-fg} we study the invariant
\[
c(R, S) = \max \{ r \mid H^{\dim R - i}_{I}(R)_n = 0 \ \text{for all $n \geq 0$ and all $i < r$} \}.
\]
In Proposition \ref{fg-prop} we show that $c(R,S) = \fg^*(S)$, an invariant of $S$. Thus $c(R,S)$ is independent of $R$.

.

\s \label{koszul-app} \textbf{Application-III:}
Let $(A,\m)$ be a equicharacteristic Noetherian local ring with infinite residue field. Assume there $\pi \colon R \rt A$ is a surjection where $R$ is a  regular local ring of dimension $d$.  Let $\m$  be the maximal ideal of $R$.
We show
\begin{theorem}\label{koszul-app-th}
(with hypotheses as in \ref{koszul-app}) The Koszul cohomology modules \\ $H^j(\m, H^{d-i}_{\ker \pi}(R))$ depends only on $A,i$ and $j$ and neither  on $R$ nor on $\pi$.
\end{theorem}

We now describe in brief the contents of this paper. In section two we discuss a few preliminary results  that we need. In section three we discuss an abstraction of a result due to Lyubeznik. In section four we prove Theorem \ref{main-koszul} when $r = 1$. In section five we prove Theorem \ref{main-koszul} in general. In section six we give proofs of Corollary \ref{cor-koszul} and \ref{cor-graded}. In section seven we give prrofs of most of application I when $A$ is an infinite  field of characteristic $p > 0$. In the next section we prove Theorem
\ref{vanish} and Theorem \ref{tame}. In section nine we prove Theorems \ref{bass-basic}, \ref{bass-m-one} and Theorem \ref{bass-growth}.  In the next section we prove Theorem  \ref{ass}. In section eleven we prove Theorem \ref{injdim-and-dim}. In section twelve we prove Theorem \ref{inf-gen}. In section thirteen we give proof of application II. In the next section we discuss application III.
\section{Preliminaries}

In this section, we discuss a few preliminary results that we need.

 \s \textbf{Lyubeznik functors:} \\
 Let $B$ be a commutative Noetherian ring and let $X = \Spec(B)$. Let $Y$ be a locally closed subset of $X$.
 If $M$ is a $B$-module and  $Y$ be a locally closed subscheme of $\Spec(R)$, we denote by $H^i_Y(M)$ the
 $i^{th}$-local cohomology module of $M$ with support in $Y$.  Suppose
 $Y = Y_1 \setminus Y_2$ where $Y_2 \subseteq Y_1$ are two closed subsets of $X$ then we have an exact sequence of functors
 \[
 \cdots \rt H^i_{Y_2}(-) \rt H^i_{Y_1}(-) \rt H^i_Y(-) \rt H^{i+1}_{Y_2}(-) \rt .
 \]
 A Lyubeznik functor $\FF$ is any functor of the form $\FF = \FF_1\circ \FF_2 \circ \cdots \circ \FF_m$ where every functor $\FF_j$  is either $H^i_Y(-)$ for some locally closed subset of $X$ or the kernel, image or
cokernel of some arrow in the previous long exact sequence for closed
subsets $Y_1,Y_2$ of $X$  such that $Y_2 \subseteq Y_1$.

\s \textit{Lyubeznik functor under flat maps:}\\
We need the following result from \cite[3.1]{Lyu-1}.
\begin{proposition}\label{flat-L}
 Let $\phi \colon B \rt C$ be a flat homomorphism of Noetherian rings. Let $\FF$ be a
 Lyubeznik functor on $Mod(B)$. Then there exists a Lyubeznik functor $\widehat{\FF}$ on $Mod(C)$ and
 isomorphisms $\widehat{\FF}(M\otimes_B C) \cong \FF(M)\otimes_B C$ which is functorial in $M$.
\end{proposition}

\s  \textbf{Graded Lyubeznik functors:} \\
Let $A$ be a commutative Noetherian ring and let $R = A[X_1,\ldots, X_m]$ be standard graded.
We say $Y$ is \textit{homogeneous }closed subset of $\text{Spec}(R)$ if
$Y= V(f_1, \ldots, f_s)$, where $f_i's$ are homogeneous polynomials in $R$.

We say $Y$ is a homogeneous locally closed subset of $\text{Spec}(R)$ if $Y=Y''-Y'$, where $Y', Y''$
are homogeneous closed subset of $\text{Spec}(R)$.  Let $\ ^*Mod(R)$ be the category of graded $R$-modules.
We have an exact sequence of  functors on $\ ^*Mod(R)$,
\begin{equation}\label{eq1} H_{Y'}^i(-) \longrightarrow H_{Y''}^i(-) \longrightarrow H_{Y}^i(-) \longrightarrow
H_{Y'}^{i+1}(-).
\end{equation}

\begin{definition}\label{defn-grade-Lyu-functor}
\textit{A graded Lyubeznik functor} $\mathcal{T}$ is a composite functor of the form
$\mathcal{T}= \mathcal{T}_1\circ\mathcal{T}_2 \circ \ldots\circ\mathcal{T}_k$, where
each $\mathcal{T}_j$ is either $H_{Y_j}^i(-)$, where $Y_j$ is a homogeneous locally closed subset of $\text{Spec}(R)$
or the kernel of any arrow appearing in (\ref{eq1}) with $Y'=Y_j'$ and $Y''= Y_j''$, where $Y_j' \subset Y_j''$ are two homogeneous  closed subsets of $\text{Spec}(R)$.
\end{definition}

\s \label{std-op} \textit{Graded Lyubeznik functors \wrt \ some standard operations on $B$.}

If $A \rt B$ is a flat map then we have a flat map of graded rings \\ $R = A[X_1, \ldots, X_m] \rt S = B[X_1,\ldots, X_m]$. If $\FF$ is a graded Lyubeznik functor
on $R$ then  there is a graded Lyubeznik functor $\wh{\FF}$  on $ \ ^*Mod(S)$  with
 $$\wh{\FF}(S) = \FF(R)\otimes_R S  =  \FF(R)\otimes_A B.$$

The specific flat base changes we do are the following;
\begin{enumerate}
 \item $B = A_P$ where $P$ is a prime ideal of $A$.
 \item If $(A,\m)$ is local then  we take $B$ to be the $\m$-adic completion of $A$.
 \item If $(A,\m)$ is local with finite residue field then take $B = A[X]_{\m A[X]}$.
\end{enumerate}

\s \label{min-loc} We will use the following well-known result often. Let $B$ be a Noetherian ring and let $M$ be an
$A$-module not necessarily finitely generated. Let $P$ be a minimal prime of $M$. Then the $B_P$-module $M_P$ has a natural structure
of an $\widehat{B_P}$-module (here $\widehat{B_P}$ is the completion of $B_P$ \wrt \ it's maximal ideal $PB_P$). In fact
$M_P \cong  M_P\otimes_{B_P}\widehat{B_P}$.

\s Let $A$ be a Noetherian ring, $I$ an ideal in $A$ and let $M$ be an $A$-module, not necessarily finitely generated.
Set
\[
\Gamma_I(M) = \{ m \in M \mid I^sm = 0 \ \text{for some} \ s \geq 0 \}.
\]
 The following result is well-known. For lack of a suitable reference we give sketch of a proof here. When $M$ is finitely generated, for a proof of the following result see \cite[Proposition 3.13]{E}.
\begin{lemma}\label{mod-G}[with hyotheses as above]
\[
\Ass_A \frac{M}{\Gamma_I(M)} = \{ P \in \Ass_A M \mid P \nsupseteq I \}
\]
\end{lemma}
\begin{proof}\textit{(sketch)}
Note that if $P \in \Ass_A \Gamma_I(M)$ then $P \supseteq I$. It follows that if $P \in \Ass_A M$ and $P \nsupseteq I$
then $P \in \Ass_A M/\Gamma_I(M)$.

It can be easily verified that if  $P \in \Ass_A M/\Gamma_I(M)$ then $P \nsupseteq I$. Also note that if $P \nsupseteq I$ then $\Gamma_I(M)_P = 0$. Thus
\[
M_P \cong \left(\frac{M}{\Gamma_I(M)}\right)_P  \quad \text{if} \ P \nsupseteq I.
\]
The result follows.
\end{proof}

\s \emph{$F$-modules.}\\
Let $R$ be a regular ring of characteristic $p > 0$.  The concept of $F_R$-modules was introduced in \cite{Lyu-2}. We recall the notion here.

Let $R'$ to be $R$-bimodule which agrees with $R$ as a left $R$-module, and has the right $R$-action
\[
r'r = r^pr \ \quad \text{for} \ r \in R \ \text{and} \ r' \in R'.
\]
For an $R$-module $M$ set $F(M) = R'\otimes_R M$. This is a $R$-module via the left $R$-module structure on $R'$.
An $F_R$-module is a pair $(\M,\theta)$ where $\M$ is an $R$-module and $\theta \colon \M \rt F(\M)$ is a $R$-module homomorphism. We usually suppress $\theta$ from the notation.

\s \emph{Graded $F$-modules}
Let $A$ be a regular ring containing a field of characteristic $p > 0$.
Let $R = A[X_1,\ldots, X_m]$ be standard graded. If $M$ is a graded $R$-module, and $m \in M$ is homogeneous set $|m| = \deg(m)$.
Graded $F_R$-modules have been studied previously in \cite[4.3.3]{BM} and  \cite{Z}.
If $M$ is a graded $R$-module then there is a natural grading on $F(M) = R'\otimes_R M$ defined by
\[
|r'\otimes m| = |r'| + p |m|
\]
for homogeneous elements $r' \in R'$ and $m \in M$.
With this grading, a graded $F_R$-module is
an $F_R$-module $(\M, \theta )$ where  $\M$ is a graded $R$-module, and $\theta$ is degree-preserving, i.e., $\theta$
maps homogeneous elements to homogeneous elements of the same degree.

The ring $R$ has a natural graded $F_R$-module structure  with structure morphism
\[
R \rt R'\otimes_R R  \quad  \quad  r \rt r\otimes 1.
\]

\s \emph{$F_R$-finite modules.}\\
A $F_R$-module $\M$ is said be $F_R$-finite if $\M$ is a direct limit of the top row in the commutative diagram,
\[
  \xymatrix
{
 M
 \ar@{->}[r]^{\beta}
 \ar@{->}[d]^{\beta}
  & F(M)
\ar@{->}[r]^{F(\beta)}
\ar@{->}[d]^{F(\beta)}
 & F^2(M)
\ar@{->}[r]
\ar@{->}[d]^{F^2(\beta)}
& \cdots
\\
 F(M)
 \ar@{->}[r]^{F(\beta)}
  & F^2(M)
\ar@{->}[r]^{F^2(\beta)}
 & F^3(M)
\ar@{->}[r]
& \cdots
\
 }
\]
where $M$ is a finitely generated $R$-module, $\beta \colon M \rt F(M)$ is an $R$-module homomorphis and the structure isomorphism $\theta$ is induced by the vertical maps in the diagram
(see \cite[2.1]{Lyu-2}). When $M$ is graded and $\beta$ is degree-preserving, we say that the
$F_R$-module $\M$ is graded $F_R$-finite.

The map $\beta \colon M \rt F(M)$  above is a generating morphism of $\M$. If $\beta$ is injective,
we say that $M$ is a root of $\M$, and that $\beta$ is a root morphism.

\s If $\FF $ is a Lyubeznik functor on $R$ (with $R$-regular of characteristic $p>0$) then $\FF(R)$ is $F_R$-finite, see \cite[2.14.]{Lyu-2}. A similar argument shows that if $R$ is graded and $\FF$ is a graded Lyubeznik functor on $R$ then $\FF(R)$ is graded $F_R$-finite.

\s \label{bad-de-Rham} Let $k$ be a field of  characteristic $p > 0$. Set $R = k[X]$. Let $\partial = \partial/\partial X$. Then note
the de-Rham cohomology module
\[
H^0(\partial, R) = \oplus_{ p|n} kX^n.
\]
is not finite dimensional over $k$.
\section{An abstraction of a result due to Lyubeznik}
In this section we indicate an abstraction of a result of Lyubeznik from \cite{Lyu-2}. We first indicate Lyubeznik's result.

\s (\cite[1.3]{Lyu-2})   Let $R, S$ be regular rings of characteristic $p > 0$ and let $\pi \colon R \rt S$ be a ring homomorphism. Let
$\pi^\prime_*  \colon \text{Mod}(R) \rt \text{Mod}(S)$ be the functor $S\otimes_R -$. Then Lyubeznik constructs an isomorphism of functors
$\phi \colon \pi^\prime_* \circ F_R \rt F_S\circ\pi^\prime_*$. Then he constructs a functor
\begin{align*}
  \pi_* &\colon \text{$F_R$-Mod} \rt  \text{$F_S$-Mod} \\
  \pi_*(\M, \theta) &= (\pi_*^\prime(\M), \phi\circ \pi_*^\prime(\theta)) \\
  \pi_*(f) &= \pi_*^\prime(f).
\end{align*}
In \cite[2.9]{Lyu-2} he shows that $\pi_*$ takes $F_R$-finite modules to $F_S$-finite modules.

\s Let $G \colon \text{$R$-Mod} \rt \text{$S$-Mod}$ be an additive functor. Let $\{ M_\alpha \}_{\alpha \in \Gamma}$ be a direct system of $R$-modules.
Then note that we have a map of $S$-modules
$$ \psi_G \colon \lim_{\alpha \in \Gamma} G(M_\alpha)   \rt G(\lim_{\alpha \in \Gamma} M_\alpha).$$
If for all choices of direct systems the above map is an isomorphism then we say $G$ commutes with direct limits.

\s\label{close}  A close inspection of \cite[1.3(a), 1.10(g), 2.9(a)]{Lyu-2} shows that Lyubeznik only uses the following three properties of  $\pi^\prime_*$;
\begin{enumerate}
  \item We have an isomorphism of functors $\phi \colon \pi^\prime_* \circ F_R \rt F_S\circ\pi^\prime_*$.
  \item $\pi^\prime_*$ commutes with direct limits.
  \item If $E$ is a finitely generated $R$-module then $\pi^\prime_*(E)$ is a finitely generated  $S$-module.
\end{enumerate}
Thus we have the following result:
\begin{theorem}\label{formal}
   Let $R, S$ be regular rings of characteristic $p > 0$ and let $\pi \colon R \rt S$ be a ring homomorphism. Let
$\eta  \colon \text{Mod}(R) \rt \text{Mod}(S)$ be an additive functor. Suppose
\begin{enumerate}[\rm (i)]
  \item we have an isomorphism of functors
$\phi \colon \eta \circ F_R \rt F_S\circ\eta$.
  \item $\eta$ commutes with direct limits.
  \item  if $E$ is a finitely generated $R$-module then $\eta(E)$ is a finitely generated  $S$-module.
\end{enumerate}
Then we have a functor
\begin{align*}
  \eta_* &\colon \text{$F_R$-Mod} \rt  \text{$F_S$-Mod} \\
  \eta_*(\M, \theta) &= (\eta(\M), \phi\circ \eta(\theta)) \\
  \eta_*(f) &= \eta(f).
\end{align*}
Furthermore $\eta_*$ takes $F_R$-finite modules to $F_S$-finite modules.
\end{theorem}
\begin{proof}
  The proof follows \cite[1.3(a), 1.10(g), 2.9(a)]{Lyu-2}  verbatim. So it is omitted.
\end{proof}
We will use Theorem \ref{formal} in the next section.

\section{Proof of Theorem \ref{main-koszul} when $r = 1$}
In this section we give a proof of Theorem \ref{main-koszul} when $r = 1$.  In fact we  prove a considerably stronger statement. We also indicate how \emph{we cannot} prove Theorem \ref{main-koszul} by induction on number of variables using Theorem \ref{r1}.
\begin{theorem}\label{r1}
  Let $R$ be a regular ring of characteristic $p > 0$. Let $X$ be a $R$-regular element such that $R/XR$ is regular.
  Let $\M$ be a $F_R$-finite module. Then the Koszul homology modules $\M/X \M$ and $(0 \colon_\M X)$ are $F_{R/(X)}$-finite.
\end{theorem}
We note that the fact $\M/X \M$ is $F_{R/(X)}$-finite follows from \cite[1.3, 2.9]{Lyu-2}. We first need the following result:
\begin{lemma}\label{omega}
(with hypotheses as in \ref{r1}). The map
\begin{align*}
  R/(X)  &\xrightarrow{\alpha} \Hom_R(R/(X), R/(X^p)), \\
  a + (X) &\mapsto X^{p-1}a + (X^p),
\end{align*}
is an isomorphism.
\end{lemma}
\begin{proof}
It is evident that $\alpha$ is well defined and injective.  To show $\alpha$ is surjective it suffices to show $\alpha \otimes_R \wh{R_\m}$ is surjective for every maximal ideal $\m$ of $R$ containing $(X)$. As $R/XR$ is regular we get $\wh{R_\m}/X\wh{R_\m}$- is regular. It follows that $X$ is a regular parameter in $\wh{R_\m}$. So we can assume
$\wh{R_\m} = \kappa(\m)[[X= X_1, X_2, \cdots, X_n]]$.  In this case it is clear that $\alpha \otimes_R \wh{R_\m}$ is surjective.
\end{proof}
\begin{remark}\label{lyu-w}
(with hypotheses as in \ref{r1}).
In \cite[3.1]{Lyu-2} Lyubeznik constructs a functorial isomorphism for each  $R/(X)$-module $\M$,
\[
\Hom_R(R/(X), F_R(\M)) \rt \omega^{p-1}\otimes_{R/(X)}F_{R/(X)}(\M).
\]
Taking $\M = R/(X)$ we get
\[
\omega^{p-1} \cong \Hom_R(R/(X), F_R(R/(X))) \cong  \Hom_R(R/(X), R/(X^p)) \cong R/(X),
\]
where the last isomorphism follows from Lemma \ref{omega}.
So we have a natural isomorphism for each  $R/(X)$-module $\M$,
\[
\delta \colon \Hom_R(R/(X), F_R(\M)) \rt F_{R/(X)}(\M).  \tag{$\dagger$}
\]
\end{remark}
\begin{construction}\label{const}
(with hypotheses as in \ref{r1}). Consider the additive functor $\eta  \colon \text{Mod}(R) \rt \text{Mod}(R/(X))$ defined as $\psi(\M) = (0 \colon_\M X)$. If $f \colon \M \rt \N$ then  note $f(0\colon_\M X) \subseteq (0 \colon_\N X)$. Define $\eta(f)$ to be this restricted map.
Consider the following sequence of functorial isomorphisms,
\begin{align*}
  F_{R/(X)}(\eta(\M)) &\xrightarrow{\delta^{-1}} \Hom_R(R/(X), F_R(0 \colon_\M X)), \\
   &= \Hom_R(R/(X), (0 \colon_\M X^p)), \\
   &\cong \Hom_R(R/(X), \Hom_R(R/(X^P), F_R(\M)), \\
  &\cong \Hom_R(R/(X)\otimes_R R/(X^p), F_R(\M)) \\
   &= \Hom_R(R/(X), F_R(\M)), \\
   &\cong \eta(F_R(M)).
\end{align*}
\end{construction}
We now give a proof of,
\begin{proof}[Proof of Theorem \ref{r1}]
As discussed earlier $\M/X M$ is $F_{R/(X)}$-finite. Next we prove $(0 \colon_\M X)$ is  $F_{R/(X)}$-finite.
To this note that the functor $\eta$ defined in \ref{const} yields $\eta(\M) =  (0 \colon_\M X)$.

We also have a functorial isomorphism $\phi \colon \eta\circ F_R \rt F_{R/(X)} \circ \eta$ (this is inverse of the isomorphism constructed in \ref{const}).

We note that $\eta \cong \Hom_R(R/(X), -)$ commutes with direct limits, see \cite[3.4.3]{BS}. Furthermore if $E$ is a finitely generated $R$-module then $\eta(E)$ is a finitely generated $R/(X)$-module.

The result follows from Theorem \ref{formal}.
\end{proof}

\begin{remark}
We now explain why we cannot prove Theorem \ref{main-koszul} by inducting on the number of variables.
Let $X, Y$ be a regular sequence in $R$ with $R/(X)$ and $R/(X, Y)$ regular rings. Then we have a short exact sequence,
\[
0 \rt H_0(Y;  H_1(X ; \M)) \rt H_1(X,Y; \M) \rt H_1(Y;  H_0(X; \M)) \rt 0.
\]
The modules at both the ends are $F_S$-finite where $S = R/(X, Y)$. \emph{However we do not know whether the  above short exact sequence is a sequence in the category of $F_S$-modules}.
So we cannot conclude that $ H_1(X,Y; \M)$ is $F_S$-finite.
\end{remark}
We however have the following corollary to Theorem \ref{r1}.
\begin{theorem}\label{top}
    Let $R$ be a regular ring of characteristic $p > 0$. Let $X_1, \ldots, X_n$ be a $R$-regular sequence  such that $R/(X_1, \ldots, X_i)R$ is regular for $i = 1, \ldots, n$.
  Let $\M$ be a $F_R$-finite module. Then the Koszul homology module $H_n(X_1, \ldots, X_n; M)$ is $F_{S}$-finite where $S =R/(X_1, \ldots, X_n)R$.
\end{theorem}
\begin{proof}
  We prove the result by induction on $n$. When $n = 1$ the result follows from Theorem \ref{r1}. Assume the result holds for $n-1$, i.e., $H_{n-1}(X_1, \ldots, X_{n-1}; \M)$ is $F_T$-finite where
  $T =R/(X_1, \ldots, X_{n-1})R$.

  Note we have an isomorphism of $S$-modules
  \[
  H_n(X_1, \ldots, X_n; \M) \cong H_1(X_n, H_{n-1}(X_1, \ldots, X_{n-1}; \M)).
  \]
  We apply Theorem \ref{r1} to conclude.
\end{proof}

\section{Proof of Theorem \ref{main-koszul}}
In this section we prove Theorem \ref{main-koszul}. We restate it for the convenience of the reader.
\begin{theorem}\label{main-koszul-body}
Let $k$ be an infinite field of characteristic $p > 0$. Let $R$ be one of the following regular rings
\begin{enumerate}[\rm (i)]
  \item $k[Y_1, \ldots, Y_d]$.
  \item $k[[Y_1,\ldots, Y_d]]$.
  \item $A[X_1, \ldots, X_m]$ where $A = k[[Y_1,\ldots, Y_d]]$.
\end{enumerate}
Let $\M$ be a $F_R$-finite module. Fix $r \geq 1$. Then the Koszul homology   \\ modules $H_i(Y_1,\ldots, Y_r; \M)$ are
$F_{\ov{R}}$-finite  modules (where $\ov{R} = R/(Y_1,\ldots, Y_r)$) for $i = 0, \ldots,r$.
\end{theorem}
We prove the result by induction on $r$.  When $r = 1$ the result holds by Theorem \ref{r1}.
\begin{lemma}\label{base-koszul}
(with hypotheses as in \ref{main-koszul-body}). The result holds for $r = 1$. \qed
\end{lemma}

The following result is a crucial ingredient to prove Theorem \ref{main-koszul-body}.
\begin{lemma}\label{m-torsion}
(with hypotheses as in \ref{main-koszul-body}). Set $\N = \Gamma_{(Y_1,\ldots, Y_r)}(\M)$. Then \\ $H_i(Y_1,\ldots, Y_r; \N) = 0$ for $i < r$.
\end{lemma}
We now give a proof of Theorem \ref{main-koszul-body} assuming Lemma \ref{m-torsion}.
\begin{proof}[Proof of Theorem \ref{main-koszul-body}]
We prove the result by induction on $r$. The base case $r = 1$ is dealt with in Lemma \ref{base-koszul}.
Assume the result for $r = n -1 \geq 1$. We prove the result for $r = n$.

The case when $i = n$ follows from Theorem\ref{top}.

We have an exact sequence of $F_R$-finite modules:
$$0 \rt \Gamma_{(Y_1,\ldots, Y_n)}(\M) \rt \M \rt \ov{\M} \rt 0.$$

Case 1. $\ov{\M} = 0$.

In this case we have $\Gamma_{(Y_1,\ldots, Y_n)}(\M) = \M$.  By Lemma \ref{m-torsion} it follows that \\ $H_i(Y_1, \ldots, Y_n; \M ) = 0$ for  $i < n$.  Thus in this case we have the result.

Case 2. $\ov{\M} \neq 0$.

We note that by our construction $H_i(Y_1,\ldots, Y_n, \M) = H_i(Y_1,\ldots, Y_n, \ov{\M})$ for $i < n$. Thus it suffices to prove the result for $\ov{\M}$.

As $\ov{\M}$ is $F_R$-finite we have that $\Ass_R \ov{\M}$ is a finite set. If $P \in \Ass_R \ov{\M}$ then note  $P \nsupseteq (Y_1, \ldots, Y_n)$.
Let $V = kY_1\oplus \cdots \oplus kY_n$ be the  finite-dimensional $k$-vector subspace of $R$  with basis $Y_1, \cdots, Y_n$. We have
$P\cap V$ is a proper subspace of $V$ for every $P \in \Ass_R \ov{\M}$. As $k$ is infinite and as $\Ass_R \ov{\M}$ is a finite set there exists
$$ Z \in V \setminus \bigcup_{P \in \Ass_R \ov{\M}} P\cap V. $$
Say $Z = \sum_{i = 1}^{n} \alpha_i Y_i$ with $\alpha_i \in k$ and some $\alpha_j \neq 0$. Without loss of generality we can assume $\alpha_1 \neq 0$.
We consider the change of variables,
\begin{align*}
  Z_1 &=  Z, \ \text{and} \\
  Z_j &= Y_j \ \text{for $j > 1$}.
\end{align*}
Then note $H_*(Z_1, \ldots, Z_n, -) \cong H_*(Y_1, \ldots, Y_n, -)$. Furthermore we have
\begin{enumerate}
  \item if $R = k[Y_1, \ldots, Y_d]$ then $R = k[Z_1, \ldots, Z_d]$. Also $R/Z_1R = k[Z_2, \ldots, Z_d]$.
  \item if $R = k[[Y_1, \ldots, Y_d]]$ then $R = k[[Z_1, \ldots, Z_d]]$. Also \\ $R/Z_1R = k[[Z_2, \ldots, Z_d]]$.
  \item if $R = A[X_1, \dots, X_m]$ where $A = k[[Y_1, \ldots, Y_d]]$ then \\  $R/Z_1R = B[X_1, \dots, X_m]$ where $B = k[[Z_2, \ldots, Z_d]]$.
\end{enumerate}
We note that $Z_1$ is $\ov{\M}$-regular. So we have an isomorphism
$$ H_*(Z_1, \ldots, Z_n, \ov{\M})  \cong H_*(Z_2, \ldots, Z_n, \ov{\M}/Z_1\ov{\M}). $$
By \ref{r1},  $\ov{\M}/Z_1\ov{\M}$ is $F_{R/Z_1R}$-finite. By induction hypotheses we get that \\ $H_*(Z_2, \ldots, Z_n, \ov{\M}/Z_1\ov{\M})$ are $F_S$-finite
where
$$S = R/(Z_1,\ldots, Z_n) = R/(Y_1, \ldots, Y_n).$$
 The result follows.
\end{proof}
We now give:
\begin{proof}[Proof pf Lemma \ref{m-torsion}]
Claim-1: We may reduce to the case when \\
$R = k[[Y_1, \ldots, Y_r, Z_1, \ldots, Z_s]]$.

Fix $i < r$. Let $E = H_i(Y_1, \ldots, Y_r, \N)$. Then notice $E$ is $(Y_1, \ldots, Y_r)$-torsion. Then $E$ is zero if and only if $L = E\otimes_{R} R_\m = 0$ for all maximal ideals $\m \supseteq (Y_1, \ldots, Y_r)$. Note $L_\m =  H_i(Y_1, \ldots, Y_r, \N_\m)$. By
\cite[1.3, 2.9]{Lyu-2} we have that $\N_m$ is a $F_{R_\m}$-finite.  Furthermore $L$ is zero if and only if $L\otimes_{R_\m} \wh{R_\m} = 0$ where $\wh{R_\m}$ is the completion of $R_\m$. Notice $L\otimes_{R_\m} \wh{R_\m} = H_i(Y_1, \ldots, Y_r, (\N_\m)\otimes_{R_\m} \wh{R_\m})$. By  \cite[1.3, 2.9]{Lyu-2} we have that $\N_m \otimes_{R_\m} \wh{R_\m}$ is a $F_{\wh{R_\m}}$-finite.  We note that $R/(Y_1, \ldots, Y_r)$ is regular. So
$\wh{R_\m}/(Y_1, \ldots, Y_r)\wh{R_\m}$ is regular. Thus $Y_1, \ldots, Y_r$ is part of a regular system of parameters of the complete regular ring $\wh{R_\m}$. It follows that \\
$\wh{R_\m} \cong \kappa(\m)[[Y_1, \ldots, Y_r, Z_1, \ldots, Z_s]]$ for some $Z_1,\ldots, Z_s$.
Thus Claim-1 is true.

Thus $R = k[[Y_1, \ldots, Y_r, Z_1,\ldots, Z_s]]$. We prove our result by induction on $r$.
We first consider the case when $r = 1$.  For convenience put $Y =  Y_1$. Let $\D$ be the ring of $k$-linear differential operators on $R$. By
\cite[5.1]{Lyu-2} $\N$ is a $\D$-module.
Let $t \in N$. Then $Y^ct = 0$ for some $c \geq 1$. Put
$$ \partial_{[m]} = \frac{1}{m!}\frac{\partial^m}{\partial Y^m}. $$
For all $m \geq 1$ we have by \cite[2.2]{Lyu-3},
$$  \partial_{[m]} Y^m = \sum_{i= 0}^{m} \binom{m}{i}Y^{m-i}\partial_{[m-i]}.$$
So we have
$$ 0 = \partial_{[c]}Y^c t =  \big{ \{}\sum_{i= 0}^{c-1} \binom{c}{i}Y^{c-i}\partial_{[c-i]}t \big{ \} } + t.$$
It follows that $t \in Y\N$. Thus $\N = Y \N$. So $H_0(Y, \N) = 0$.

Now assume $r \geq 2$ and the result has be proved for $r-1$. Set $\bY = Y_1, \ldots, Y_r$ and $\bY^\prime = Y_1, \ldots, Y_{r-1}$.
We have an exact sequence for all $i$
\[
0 \rt H_0(Y_r, H_i(\bY^\prime, \N)) \rt H_i(\bY, \N) \rt H_1(Y_r, H_{i-1}( \bY^\prime,  \N)) \rt 0.
\]
By our induction hypotheses $H_i(\bY^\prime, \N) = 0$ for $i < r-1$. So $H_i(\bY, \N) = 0$ for $i < r -1$ and
$$ H_{r-1}(\bY, \N) =  H_0(Y_r, H_{r -1}(\bY^\prime, \N)). $$
Note that  $E = H_{r-1}(\bY^\prime, \N))$ is a $F_S$-finite module where $S = R/\bY^\prime R$. Also it is easy to check that $\Gamma_{Y_r}(E) = E$. So by our case when $r = 1$ we get
$H_0(Y_r, E) = 0$. Thus $ H_{r-1}(\bY, \N) =  0$. Our result follows by induction.
\end{proof}
We note that Lemma \ref{m-torsion} is essentially characteristic free. So we have
\begin{lemma}\label{m-torsion-char 0}
Let $k$ be  field of characteristic $ 0$. Let $R$ be one of the following regular rings
\begin{enumerate}[\rm (i)]
  \item $k[Y_1, \ldots, Y_d]$.
  \item $k[[Y_1,\ldots, Y_d]]$.
  \item $A[X_1, \ldots, X_m]$ where $A = k[[Y_1,\ldots, Y_d]]$.
\end{enumerate}
Let $\FF$ be a Lyubeznik functor on $Mod(R)$ and let $\M = \FF(R)$. In case (i) or (ii) let $\M$  be any holonomic $\D$-module (here $\D$ is the ring of $k$-linear differential operators on $R$). Fix $r \geq 1$. Set $\N = \Gamma_{(Y_1,\ldots, Y_r)}(\M)$. Then  $H_i(Y_1,\ldots, Y_r; \N) = 0$ for $i < r$.
\end{lemma}
\begin{remark}
\label{grading} We can consider graded analogue of Theorem \ref{main-koszul-body}. For this let
\begin{enumerate}
  \item $R = k[Y_1,\ldots, Y_d]$ where $\deg Y_i = 1$ for all $i$.
  \item $R = A[X_1,\ldots, X_m]$ where $A = k[[Y_1, \ldots, Y_d]]$. Consider $R$ graded with $\deg A = 0$ and $\deg X_i = 1$ for all $i$.
\end{enumerate}
Let $\M$ be a graded $F_R$-finite module. Then $H_*(Y_1, \ldots, Y_r; \M)$ are graded $F_S$-finite modules where $S = R/(Y_1,\ldots, Y_r)$.

To see this observe the graded analogue of Theorems \ref{formal} holds. Rest of the arguments are similar. Note we only make homogeneous change of variables.
\end{remark}

\section{Proof of Corollary \ref{cor-koszul} and \ref{cor-graded}}
In this section we give proofs of Corollary \ref{cor-koszul} and \ref{cor-graded} which are needed for our applications. We first prove
\begin{proof}[Proof of Corollary \ref{cor-koszul} ]
Fix $i$. Set $\N = H_i(X_1,\ldots, X_d; \M)$. Then $\N$ is  $F_k$-finite.

As $\N$ is $F_k$-finite it has a root say $U$ which is a finite dimensional $k$-vector space.
We have chain of injective $k$-linear maps
\[
U \hookrightarrow F(U) \hookrightarrow F^2(U) \hookrightarrow \cdots
\]
 with $\N$ being the direct limit. But $U$ is a finite dimensional $k$-vector space. So $U = k^r$ for some $r$. Then $F(U) \cong k^r$. As $U \hookrightarrow F(U)$ is an inclusion we get that it is an isomorphism as they are vector spaces of same dimension. Similarly we can show that  each map in the above inductive system is an isomorphism. So $\N \cong U = k^r$. Thus $\N$ is a finite dimensional $k$-vector space.
\end{proof}

\s\label{MaZ}  Let $R = k[X_1,\ldots, X_d]$ be standard graded. Let $\N$ be a graded $F_R$-finite module. Suppose $\N$ is supported only at $(X_1,\ldots, X_m)$.
So it is a finite direct sum of $E(k)$ possibly with some shifts (here $E(k)$ is the graded injective hull of $k$).  By \cite[5.6]{MaZhang} all the socle elements are in degree $-d$. So $\Gamma_\m (\M) = E^s(d)$ for some $s \geq 0$.

Next we give a proof of
\begin{proof}
[Proof of Corollary \ref{cor-graded}]
Set $\m = (X_1, \ldots, X_d)$.
We have an exact sequence
\[
0 \rt \Gamma_\m(\M) \rt \M \rt \ov{\M} \rt 0.
\]
$\Gamma_\m (\M)$ is a  graded $F_R$-finite module supported only at the graded maximal ideal $\m$ of $R$.  So $\Gamma_\m (\M) = E^s(d)$. It is well-known and easy to prove that
\begin{align*}
  H_i(X_1, \ldots, X_d; E(d)) &= 0 \ \text{for} \ i < d \\
    &=k  \ \text{(concentrated in degree zero)} \ \text{for} \ i = d.
\end{align*}
As in the proof of Theorem \ref{main-koszul-body} we have $H_d(X_1, \ldots, X_d; \ov{\M}) = 0$. Thus \\  $H_d(X_1, \ldots, X_d; \M)$ is concentrated in degree  zero.
We also have \\  $H_i(X_1, \ldots, X_d; \M) \cong H_i(X_1,\ldots, X_d, \ov{\M})$ for $i< d$.

As in the proof of Theorem \ref{main-koszul-body} we have a linear homogeneous change of variables $Z_1, \ldots, Z_d$ such that
\[
 H_*(Z_1, \ldots, Z_d; \ov{\M})  \cong H_*(Z_2, \ldots, Z_d; \ov{\M}/Z_1\ov{\M}).
\]
Thus we can reduce to the case when $d = 1$.

Set $X =  X_1$. As $\M$ is a graded $F_\R$-finite module, it is Eulerain, see \cite[4.4]{MaZhang} for this notion.
For an integer $a$ and a non-negative integer $b$ define
\[
\binom{a}{b} = \frac{a(a-1)\cdots(a -b + 1)}{b!}.
\]
Suppose $m \in \M$ is homogeneous of degree $|m| \neq 0$. Say $|m| = p^st$ where $p$ does not divide $t$.
Let
 $$\mathcal{E}_{p^s} = X^{p^s} \partial_{p^s} \quad \text{where} \ \partial_{p^s} = \frac{1}{p^s}\frac{\partial^{p^s}}{\partial X^{p^s}}.  $$
 As $\M$ is Eulerian we have
 $$ \mathcal{E}_{p^s}m = \binom{p^st}{p^s}m. $$
 It can be readily verified that $p$ does not divide $\binom{p^st}{p^s}$. It follows that $m \in X \M$. Thus $H_0(X, \M) = \M/X\M$ is concentrated in degree zero.
\end{proof}

\section{The case when $A=k$ an infinite field of characteristic $p > 0$}
In this section we prove Thoerems \ref{vanish}, \ref{tame}, \ref{bass-basic}, \ref{bass-m-one} and \ref{bass-basic} when $A = k$ is an infinite field of characteristic $p>0$.
\begin{theorem}\label{vanishing-std}
  Let $R = k[X_1,\ldots, X_m]$ be standard graded with $k$ an infinite field of characteristic $p > 0$. Let $\M = \bigoplus_{n \in \Z} \M_n$ be a graded $F_R$-finite module.
   If $\M_n = 0$ for all $|n| \gg 0$ then $\M = 0$.
\end{theorem}
\begin{proof}
Suppose if possible $\M \neq 0$. Say $\M = \bigoplus_{n = r}^{s} \M_n$  and assume $\M_s \neq 0$.
  Note that $(X_1, \ldots, X_m)\M_s = 0$.  So $\M_s \subseteq H_m(X_1, \ldots, X_m)$. The later module is concentrated in degree zero, by \ref{cor-graded}. So $s = 0$.

  Consider the submodule $\N = \Gamma_{(X_1, \ldots, X_m)}(\M)$ of $\M$. If $\N \neq 0$ then as it is   a finite direct sum of $E(k)(m)$,  it will follow $\N_j \neq 0$ for $j \leq -m$ and so $\M_n \neq 0$ for $n \leq -m$, a contradiction.  Thus $\N = 0$. As $\M$ is $F_R$-finite it has only finitely many associated primes (all of which are graded, further no associated prime is
   $(X_1, \ldots, X_m)$ since $\N = 0$).
   As $k$ is infinite, after a  homogeneous linear change of variables we may assume $X_1$ is $\M$-regular.

We prove the result by induction on $m$.
We first consider the case when $m = 1$.
As $X_1$ is $\M$-regular we have an exact sequence
\[
 0 \rt \M_{n-1} \rt \M_n \rt \ov{\M}_n \rt 0.
\]
$\ov{\M}$ is concentrated in degree zero. Also $\M_1 = 0$. The above exact sequence implies $\M_0 = 0$, a contradiction. Thus $\M = 0$.

Next we assume the result for $m-1$ and prove the result for $m$.
As $X_1$ is $\M$-regular we have an exact sequence
\[
 0 \rt \M_{n-1} \rt \M_n \rt \ov{\M}_n \rt 0.
\]
So $\ov{\M}_n  = 0$ for $|n| \gg 0$. By our induction hypothesis we get $\ov{\M} = 0$. So we get $\M_0 \cong \M_1 = 0$, a  contradiction. Thus $\M = 0$.
\end{proof}

Next we prove:
\begin{theorem}\label{tame-body}
 Let $R = k[X_1,\ldots, X_m]$ be standard graded with $k$ an infinite field of characteristic $p > 0$. Let $\M = \bigoplus_{n \in \Z} \M_n$ be a graded $F_R$-finite module.
   Then we have
\begin{enumerate}[\rm (a)]
\item
The following assertions are equivalent:
\begin{enumerate}[\rm(i)]
\item
$\M_n \neq 0$ for infinitely many $n \ll 0$.
\item
There exists $r$  such that $\M_n \neq 0$ for all $n \leq r$.
\item
$\M_n \neq 0$ for all $n \leq -m$.
\item
$\M_n \neq 0$ for some $n \leq -m$.
\end{enumerate}
\item
The following assertions are equivalent:
\begin{enumerate}[\rm(i)]
\item
$\M_n \neq 0$ for infinitely many $n \gg 0$.
\item
There exists $r$  such that $\M_n \neq 0$ for all $n \geq r$.
\item
$\M_n \neq 0$ for all $n \geq 0$.
\item
$\M_n \neq 0$ for some $n \geq 0$.
\end{enumerate}
\end{enumerate}
\end{theorem}
\begin{proof}
  (a) Notice (iii) $\implies$ (ii) $\implies$ (i)  $\implies$ (iv). So it suffices to show (iv) $\implies$ (iii).

  Suppose $\M_n \neq 0$ for some $n \leq -m$. Let $\N = \Gamma_{(X_1, \ldots, X_m)}(\M)$.

  If $\N \neq 0$ then $\N = E(m)^s$ for some $s \geq 1$. Then $\N_j \neq 0$ for all $j\leq -m$. As $\N$ is a submodule of $\M$ it follows that $\M_n \neq 0$ for all $n \leq -m$.

  If $\N = 0$ then as before (after a linear homogeneous change of variables) we can assume $X_1$ is $\M$-regular.

  We prove the result by induction on $m$.

  We first consider the case when $m = 1$.
  As $X_1$ is $\M$-regular we have an exact sequence
  \[
  0 \rt \M_{n-1} \rt \M_{n} \rt \ov{\M}_n \rt 0.
  \]
  As $\ov{\M}$ is concentrated in degree zero we get $\M_j \cong \M_{-1}$ for all $j \leq -1$. The result follows.

  We assume the result when $m = r -1 \geq 1$ and prove the result when $m = r$.
  As $X_1$ is $\M$-regular we have an exact sequence
  \[
  0 \rt \M_{n-1} \rt \M_{n} \rt \ov{\M}_n \rt 0.
  \]
  If $\ov{\M}_n \neq 0$ for some $n \leq -r +1$ then by our induction hypothesis $\ov{\M}_n \neq 0$ for all $n \leq -r + 1$. So $\M_n \neq 0$ for all $n \leq -r + 1$. The result follows in this case.

 If $\ov{\M}_n = 0$ for all $n\leq -r + 1$ then we have $\M_n \cong \M_{-r + 1}$ for all $n \leq -r +1$. The result follows.

(b) Notice (iii) $\implies$ (ii) $\implies$ (i)  $\implies$ (iv). So it suffices to show (iv) $\implies$ (iii).

Let $\N = \Gamma_{(X_1, \ldots, X_m)}(\M)$. As $\N_j = 0$ for $j \geq - m + 1$ we get that $\M_n \cong (\M/\N)_n$ for all $n \geq - m + 1$. Thus it suffices to prove the result for the $F_R$-finite module $\M/\N$.  Thus we may assume $\N = 0$. As  before (after a linear homogeneous change of variables) we can assume $X_1$ is $\M$-regular.

  We prove the result by induction on $m$.

  We first consider the case when $m = 1$.
    As $X_1$ is $\M$-regular we have an exact sequence
  \[
  0 \rt \M_{n-1} \rt \M_{n} \rt \ov{\M}_n \rt 0.
  \]
  As $\ov{\M}$ is concentrated in degree zero we get $\M_j \cong \M_{0}$ for all $j \geq 0$. The result follows.

  We assume the result when $m = r -1 \geq 1$ and prove the result when $m = r$.
  As $X_1$ is $\M$-regular we have an exact sequence
  \[
  0 \rt \M_{n-1} \rt \M_{n} \rt \ov{\M}_n \rt 0.
  \]

  If $\ov{\M}_n \neq 0$ for some $n \geq 0$ then by our induction hypothesis $\ov{\M}_n \neq 0$ for all $n \geq 0$. So $\M_n \neq 0$ for all $n \geq 0$. The result follows in this case.

 If $\ov{\M}_n = 0$ for all $n\geq 0$ then we have $\M_n \cong \M_{-1}$ for all $n \geq 0$. The result follows.
\end{proof}

\begin{theorem}\label{rigid-body}
 Let $R = k[X_1,\ldots, X_m]$ be standard graded with $k$ an infinite field of characteristic $p > 0$. Assume $m \geq 2$. Let $\M = \bigoplus_{n \in \Z} \M_n$ be a graded $F_R$-finite module.
The following assertions are equivalent:
\begin{enumerate}[\rm(i)]
\item
$\M_n \neq 0$ for all $n \in \Z$.
\item
$\M_n \neq 0$ for some $n $ with $-m < n < 0$.
\end{enumerate}
\end{theorem}
\begin{proof}
Clearly (i) $\implies$ (ii).
We prove the converse.
Let $\N = \Gamma_{(X_1, \ldots, X_m)}(\M)$. As $\N_j = 0$ for $j \geq - m + 1$ we get that $\M_n \cong (\M/\N)_n$ for all $n \geq - m + 1$. Thus it suffices to prove the result for the $F_R$-finite module $\M/\N$,  for if we show $(\M/\N)_j \neq 0$ for all $j \in \Z$ it follows that $\M_j \neq 0$ for all $j \in \Z$.
 Thus we may assume $\N = 0$. As  before (after a linear homogeneous change of variables) we can assume $X_1$ is $\M$-regular.
As $X_1$  is $\M$-regular we have an exact sequence
\[
0 \rt \M_{n-1} \rt \M_n \rt \ov{\M}_n \rt 0. \tag{$\dagger$}
\]
We prove the result by induction on $m$.

 We first consider the case when $m = 2$. We have $M_{-1} \neq 0$.
 \begin{enumerate}[\rm(a)]
   \item If $\ov{\M}_j \neq 0$ for some $j \leq -1$ then by \ref{tame-body} we have $\ov{\M}_j \neq 0$ for all $j \leq -1$. It follows that $\M_j \neq 0$  for all $j \leq -1$.
   \item If $\ov{\M}_j = 0$ for all $j \leq -1$ then by $(\dagger)$ we get $\M_j \cong \M_{-1} \neq 0$ for all $j \leq -1$. 
\item If $\ov{\M}_j \neq 0$ for some $j \geq 0$ then by \ref{tame-body} we have $\ov{\M}_j \neq 0$ for all $j \geq 0$. It follows that $\M_j \neq 0$  for all $j \geq  0$.
\item If $\ov{\M}_j = 0$ for all $j \geq 0$ then by $(\dagger)$ we get $\M_j \cong \M_{-1} \neq 0$ for all $j \geq -1$.
 \end{enumerate}
Thus $\M_j \neq 0$ for all $j \in \Z$. So we have proved the result when $m = 2$.

Assume the result holds when $m = r -1 \geq 2$. We prove the result when $m = r$.
We first prove $\M_j \neq 0$ for all $j \leq -1$.
Consider the exact sequence $(\dagger)$.
\begin{enumerate}[ \rm (1)]
  \item If $\ov{\M}_j = 0$ for some $j \leq - r + 1$ then by \ref{tame-body} we get $\ov{\M}_j = 0$ for $j \leq -r + 1$. So by our induction hypothesis $\ov{\M}_j = 0$ for $j \leq -1$.
  So we get by $(\dagger)$, $\M_j \cong \M_{-1}$ for all $j \leq -1$. Thus in this case $\M_j \neq 0$ for all $j \leq -1$
  \item If $\ov{\M}_j \neq 0$ for some $j \leq - r + 1$ then by \ref{tame-body} we get $\ov{\M}_j \neq 0$ for $j \leq -r + 1$. So by $(\dagger)$ in particular we get $\M_{-r + 1} \neq 0$.
  We consider two subcases:
  \begin{enumerate}[\rm (a)]
    \item $\ov{\M_j} \neq 0$ for some $j$ with $-r + 1 < j <0$. Then by induction hypothesis $\ov{\M}_j \neq 0$ for all $j \in \Z$. By ($\dagger$) it follows that $\M_j \neq 0$ for all $j \in \Z$.
    \item $\ov{\M_j} = 0$ for all $j$ with $-r + 1 < j <0$. By ($\dagger$) we get $\M_j = \M_{-1}$ for all $j$ with $-m  < j < 0$.
  \end{enumerate}
\end{enumerate}
Thus we have proved $\M_j \neq 0$ for all $j\leq -1$. Next we show $\M_j \neq 0$ for $j \geq 0$. We consider the exact sequence $(\dagger)$. We have two subcases:
\begin{enumerate}[\rm (1)]
  \item $\ov{\M_j} \neq 0$ for some $j \geq 0$. Then by \ref{tame-body} we get $\ov{\M}_j \neq 0$ for all $j \geq 0$. It follows $\M_j \neq 0$ for all $j \geq 0$.
  \item $\ov{\M_j} = 0$ for all $j \geq 0$. Then by $(\dagger)$ we get $\M_j \cong \M_{-1}$ for all $j \geq 0$.
\end{enumerate}
Thus the result holds by induction on $m$.
\end{proof}

For the next result we set $\ell(-) = \dim_k(-)$.
\begin{theorem}\label{len-body}
 Let $R = k[X_1,\ldots, X_m]$ be standard graded with  $k$ an infinite field of characteristic $p > 0$.  Let $\M = \bigoplus_{n \in \Z} \M_n$ be a graded $F_R$-finite module.
The following assertions are equivalent:
\begin{enumerate}[\rm(i)]
\item
$\ell(\M_n) < \infty$ for all $n \in \Z$.
\item
$\ell(\M_j) < \infty $ for some $j \in \Z $.
\end{enumerate}
\end{theorem}
\begin{proof}
Clearly (i) $\implies$ (ii). We prove the converse.
Let $\N = \Gamma_{(X_1,\ldots, X_m)}(\M)$. Then $\N$ is a finite direct sum of $E(k)(m)$. It follows that $\N_n$ has finite length for all $n \in \Z$. Further note that  $\ell(\M/\N)_j$ is finite.
Thus it is sufficient to prove $\ell(\M/\N)_n$ is finite for all $n \in \Z$. So we may assume $\N = 0$. As  before (after a linear homogeneous change of variables) we can assume $X_1$ is $\M$-regular.
As $X_1$  is $\M$-regular we have an exact sequence
\[
0 \rt \M_{n-1} \rt \M_n \rt \ov{\M}_n \rt 0. \tag{$\dagger$}
\]
We prove the result by induction on $m$.

First consider the case when $m = 1$. We note that by \ref{cor-koszul} and \ref{cor-graded} that
\begin{align*}
  H_0(\ov{\M})_s &= \ \text{finite dimensional $k$-vector space when $s =0$}, \\
  &= 0 \ \text{if $s \neq 0$}.
\end{align*}
By $(\dagger)$ it trivially  follows that $\ell(\M)_n \leq \ell(\M)_j < \infty $ for all $n \leq j$.
As $\ell(\ov{\M})_n < \infty $ for all $n$, it follows from $(\dagger)$  and a easy induction that $\ell(\M)_n  < \infty $ for all $n \geq j$.
Thus the result holds when $m = 1$.

We assume the result holds for $m = d -1 \geq 1$ and prove the result for $m = d$.
By $(\dagger)$ we get $\ell(\ov{\M}_j) < \infty$. So by induction hypotheses we get that $\ell(\ov{\M}_n) < \infty $ for all $n \in \Z$.
By $(\dagger)$ it trivially  follows that $\ell(\M)_n \leq \ell(\M)_j < \infty $ for all $n \leq j$.
As $\ell(\ov{\M})_n < \infty $ for all $n$, it follows from $(\dagger)$  and a easy induction that $\ell(\M)_n  < \infty $ for all $n \geq j$.
Thus the result holds for $m = d$. Thus by induction the Theorem is proved.
\end{proof}

We prove the following surprising result when $m = 1$.
\begin{theorem}\label{len-m-1}
 Let $R = k[X]$ be standard graded with  $k$ an infinite field of characteristic $p > 0$.  Let $\M = \bigoplus_{n \in \Z} \M_n$ be a graded $F_R$-finite module.
Then $\ell(\M_n) < \infty $ for all $n \in \Z$.
\end{theorem}
\begin{proof}
Let $\ov{k}$ denote the algebraic closure of $k$. Set $S = \ov{k}[X] = R\otimes_k \ov{k}$. Then by \cite[2.9]{Lyu-2} $\M\otimes_R S = \M \otimes_k \ov{k}$ is $F_S$-finite, Furthermore it is clear that $\dim_{\ov{k}} \M_n \otimes_k \ov{k} = \dim_k \M_n$ for all $n \in \Z$. Thus we may assume $k$ is algebraically closed.

Let $\D$ be the ring of $k$-linear differential operators on $R$.
Put
$$ \partial_{[m]} = \frac{1}{m!}\frac{\partial^m}{\partial Y^m}. $$
Then $\D = R<\partial_{[m]} \colon m \geq 1>$. We extend the grading on $R$ to $\D$ by giving $\deg \partial_{[m]} = -m$.
By \cite[5.7]{Lyu-2} we get that $\M$ has finite length as a $\D$-module. It can be easily seen by \cite[4.2]{MaZhang}    that $\M$ is a graded $\D$-module.
We note that \\ $\D_0 = k<X^i\partial_{[i]} \mid \ i \geq 1 >$.

Claim: $\M_0$ is Noetherian as a $\D_0$-module. We first note that if $V$ is a $\D_0$-submodule of $\M_0$. Then $\D V \cap \M_0 =V$.
If
$$V_1 \subseteq V_2 \subseteq \cdots \subseteq V_r \subseteq V_{r+1} \subseteq \cdots$$
be an ascending chain of $\D_0$ submodules of $\M_0$. Then we have an ascending chain of $\D$-submodules of $\M$
$$\D V_1 \subseteq \D V_2 \subseteq \cdots \subseteq \D V_r \subseteq \D V_{r+1} \subseteq \cdots$$
As $\M$ is Noetherian we get that there exists $r_0$ such that $\D V_r = \D V_{r_0}$ for all $r\geq r_0$. Intersecting with $\M_0$ we get that
$V_r = V_{r_0}$ for all $r \geq r_0$. Thus $\M_0$ is Noetherian as a $\D_0$-module. In particular it is finitely generated as a $\D_0$-module.

Say $\M_0 $ is generated as a $\D_0$ module by $u_1, \ldots, u_s$. Consider the Eulerian operator $\mathcal{E}_i = X^i\partial_{[i]}$. By \cite[4.4]{MaZhang} $\M$ is Eulerian as a $\D$-module. So  we have
\[
\mathcal{E}_i u_j = \binom{|u_j|}{i}u_j = u_j \quad \text{for} \ j = 1, \ldots, s.
\]
It follows that $\D_0 \M_0 \subseteq ku_1 + \cdots + ku_r$.
So $\ell(\M_0) < \infty$. By \ref{len-body} it follows that $\ell(\M_n) < \infty $ for all $n \in \Z$.
\end{proof}

If $\ell(\M_n)$ is finite for all $n \in \Z$ then we may ask about the functions $n \rt \ell(\M_n)$ as $n \rt \infty$ and also as $n \rt - \infty$. We show
\begin{theorem}\label{growth-body}
 Let $R = k[X_1,\ldots, X_m]$ be standard graded with  $k$ an infinite field of characteristic $p > 0$.  Let $\M = \bigoplus_{n \in \Z} \M_n$ be a graded $F_R$-finite module. Assume $\ell(\M_n) < \infty$ for all $n \in \Z$. Then there exists polynomials $P(X), Q(X) \in \mathbb{Q}[X]$ of degree $\leq m-1$ such that
\begin{enumerate}[\rm(i)]
\item
$\ell(\M_n) = P(n) $ for all $n \ll 0$.
\item
$\ell(\M_n)  = Q(n)$ for all $n \gg 0 $.
\end{enumerate}
\end{theorem}
\begin{proof}
  Consider $\N = \Gamma_{(X_1, \ldots, X_m)}(\M)$. Then as discussed above $\N = E(k)(m)^s$ for some $s \geq 0$ where $E(k)$ is the graded injective hull of  $k$. It is well-known that the assertion of the Theorem holds for $E(k)(m)$. Thus it suffices to prove the result for $\M/\N$. Thus we can assume $\N = 0$.
   As  before (after a linear homogeneous change of variables) we can assume $X_1$ is $\M$-regular.
As $X_1$  is $\M$-regular we have an exact sequence
\[
0 \rt \M_{n-1} \rt \M_n \rt \ov{\M}_n \rt 0. \tag{$\dagger$}
\]
We prove the result by induction on $m$.

We first consider the case when $m = 1$. By \ref{cor-graded}, $\ov{\M}$ is concentrated in degree zero. So we  have $\M_j  \cong \M_{-1}$ for all $j \leq -1$ and $\M_j \cong \M_0$ for all $j \geq 0$. The result follows.

We assume the result holds for $m = d -1 \geq 1$ and prove the result for $m = d$.
By $(\dagger)$  we get that $\ell(\ov{\M}_n) < \infty $ for all $n \in \Z$.  So by induction hypothesis there exists $u(X), v(X) \in \mathbb{Q}[X]$ of degree $\leq m-2$, with $\ell(\ov{\M})_n = u(n)$ for all $n \ll 0$ and
$\ell(\ov{\M}_n) = v(n)$ for all $n \gg 0$. By $(\dagger)$ we have
\begin{align*}
  \ell(\M_n) - \ell(\M_{n-1}) &= u(n) \quad \text{for $n  \ll 0$}, \\
   \ell(\M_n) - \ell(\M_{n-1}) &= v(n) \quad \text{for $n  \gg 0$}.
\end{align*}
The result follows from \cite[4.1.2]{BH}.
\end{proof}

\section{Proofs of Theorem \ref{vanish} and \ref{tame}}
In this section we give proofs of Theorems \ref{vanish} and \ref{tame}. The proof essentially involes reduction of $A$ to an infinite field of characteristic $p$ and then using results from the previous section.
\begin{lemma}
 \label{support-m}
 Let $A = k[[Y_1, \ldots, Y_d]]$ where $k$ is an infinite field of characteristic $p > 0$. Let $R = A[X_1,\ldots, X_m]$  be standard graded. Let
 $\M = \bigoplus_{n \in \Z}\M_n$ be a graded $F_R$-module. Suppose $\M_c$ is supported ONLY at the maximal ideal $\n = (Y_1, \ldots, Y_d)$ of $A$. Then
 \begin{enumerate}[\rm (a)]
   \item $\M_c = E^{\alpha}$ where $E$ is the injective hull of
$k$ as an $A$-module (here $\alpha$ is an ordinal, possibly infinite).
   \item $H_d(Y_1, \ldots, Y_d; \M)$ is a graded $F_S$-finite  module where $S = k[X_1,\ldots, X_m]$.

 \item If $\M_c \cong E_A(k)^\alpha$ then $H_d(Y_1, \ldots, Y_c, \M)_c \cong k^\alpha$.
 \end{enumerate}
 \end{lemma}
 \begin{proof}
(a)  Let $\D_k(A)$ be the ring of $k$-linear differential operators on $A$ and let $\D_k(R)$ be the ring of $k$-linear differential operators on $R$. Note $\D_k(A)$ can be considered as a subring of $\D_k(R)$. Note $\D_k(R)$ is a graded ring with  $\D_k(A) \subset \D_k(R)_0$. By \cite[5.1]{Lyu-2},  $\M$ is a graded $\D_k(R)$-module. So $\M_c$ is a $\D_k(A)$- supported only at $\n$.
 Let $T = k[Y_1,\ldots, Y_d]$. Let $\D_k(T)$ be the ring of $k$-linear differential operators on $T$. Then $\D_k(T)$ is a subring of $\D_k(A)$. So $\M_c$ is a $\D_k(T)$-module supported only at $(Y_1, \ldots, Y_d)$. So by \cite[Lemma, p. \ 208]{Lyu-inj} we get $\M_c =  E^{\alpha}$ where $E$ is the injective hull of
$k$ as an $T$-module (here $\alpha$ is an ordinal, possibly infinite). We now note that $E$ is also the injective hull of $k$ as an $A$-module.
By \ref{min-loc},  $\M_c = \M_c\otimes_T A$. It follows that $\M_c  = E^\alpha$ as an $A$-module.
The result follows.

(b) This follows from \ref{main-koszul-body}.

(c) We note that $H_d(Y_1, \ldots, Y_d; E) = k$. The result follows as
$$ H_d(Y_1, \ldots, Y_d;-) \cong \Hom_A(A/(Y_1,\ldots, Y_d),  -),$$
commutes with direct sums.
 \end{proof}

 \s \textit{Standard technique:}\label{referee}  We now discuss a basic technique which we use often. Let $A$ be a regular  ring
containing a field of characteristic $p > 0$. Set $R = A[X_1,\ldots, X_m]$ standard graded and let $\FF$ be a graded Lyubeznik functor on $ \ ^* Mod(R)$.
Set $\M = \FF(R) = \bigoplus_{n\in \Z}\M_n$. Then $\M$ is $F_R$-finite.  Suppose $\M_c \neq 0$ for some $c$. Let $P$ be a minimal prime of $\M_c$ and let $B = \widehat{A_P}$. Also
 set $S = B[X_1,\ldots, X_m]$.
 We note that by \ref{std-op} we have a $G$, a graded Lyubeznik functor on $ \ ^* Mod(S)$ with   $\N = G(S)  = B \otimes_A \FF(R) =  B\otimes_A \M$. In particular $\N$ is $F_S$-finite.

 By Cohen-structure theorem $B = K[[Y_1,\ldots, Y_g]]$ where $K = \kappa(P)$ the residue field of $A_P$  and $g = \height_A P$.  Notice
by \ref{min-loc},  $\N_c = (\M_c)_P \neq 0$. As $P$ is the minimal prime of $\M_c$ we get that $\N_c$ is supported  ONLY at the maximal ideal of $B$.
By \ref{support-m} we get that $\N_c = E_B(K)^\alpha$ for some ordinal $\alpha$ possibly infinite.

 If $K$ is finite then we take  an infinite field $K'$ containing $K$ and consider the flat extension
 $B \rt C = K'[[Y_1,\ldots, Y_d]]$. Set $T = C[X_1,\ldots, X_m]$ a flat extension of $S$. We note that by \ref{std-op}
 there is a graded Lyubeznik functor $H$ on $ \ ^* Mod(T)$
 the functor $H(T) = T \otimes_R G(S) = C \otimes_B G(S) = C\otimes_B \N$.  In particular $\LL = H(T)$    is $F_T$-finite. Notice
 $\LL_c = \N_c\otimes_B C  \neq 0$. As  $\N_c$ is supported  only at the maximal ideal of $B$ we get that $\LL_c$ is supported only at the maximal ideal of $C$.
Furthermore note that as $E_B(K) = H^g_\m(B)$ where $\m$ is the maximal ideal of $B$. Then $E_B(K)\otimes_B C = H^g_\n(C) = E_C(K')$ where $\n$ is the maximal ideal of $C$.
We note $\LL_c = E_C(K')^\alpha$.

$V = H_g(Y_1,\ldots,Y_g; \LL)$ is $F_D$-finite where   $D = K'[X_1,\ldots, X_m]$.
Furthermore $V \subseteq \LL$.
Also note that $\LL = \M \otimes_A C$.

\begin{proof}[Proof of Theorem \ref{vanish}]
 Suppose if possible $\M_c \neq 0$ for some $c$.
 We apply the standard technique \ref{referee}.
  Notice $\LL_j = C\otimes_A \M_j = 0$ for $|j| \gg 0$. Furthermore we get that
 $\LL_c \neq 0$ and supported at the maximal ideal of $C$. Furthermore $V_c \neq 0$.
As $V \subseteq \LL$ we
  we get
  $V_j = 0$ for $|j| \gg 0$.
 This contradicts  Theorem \ref{vanishing-std}. Thus $M = 0$.
\end{proof}
Next we give
\begin{proof}[Proof of Theorem \ref{tame}]
(a) Clearly (iii) $\implies$ (ii) $\implies$ (i) $\implies$ (iv). We show (iv) $\implies $ (iii).
 Suppose if possible $\M_r \neq 0$ for some $r \leq -m$.  Let $P$ be a minimal prime  of $\M_r$. We apply our standard technique \ref{referee}.
 We have $\LL_{s}  \neq 0$.
By \ref{support-m} we get that
 $V = H_g(Y_1,\ldots,Y_g; \LL)$ is $F_D$-finite where $D = K'[X_1,\ldots, X_m]$.
Furthermore $V \subseteq \LL$. We have $V_r \neq 0$.

 By Theorem \ref{tame-body}   we get  $V_n \neq 0$ for all $n \leq -m$.
 We note that $V_n \subseteq \LL_n$ for all $n \in \ZZ$. So $\LL_n \neq 0$ for all
 $n \leq -m$. As $\LL_n = \M_n \otimes_A C$ it follows that $\M_n \neq 0$ for all
 $n \leq -m$.

 (b) and (c): The proof of these assertions are similar to (a). We simply localize at a minimal prime $P$ of $\M_i$  where $i$ is appropriately  chosen and apply our standard technique \ref{referee}.
  The result then follows by Theorem \ref{tame-body} and Theorem \ref{rigid-body}.
\end{proof}

\section{Bass numbers}
In this section we give proofs of Theorems \ref{bass-basic}, \ref{bass-growth} and \ref{bass-growth}.
 \s \label{set-bass} \textit{Setup:} Let $A$ be a regular ring containing a field of characteristic $p > 0$. Let $R = A[X_1,\ldots, X_m]$ be standard graded. Let $\FF$ be a graded Lyubeznik functor on $\ ^* Mod(R)$ and set $\M = \FF(R) = \bigoplus_{n \in \ZZ}\M_n$.

By example \cite[7.4]{P} it is possible that $\mu_i(P, M_n)$  the $i^{th}$-Bass number of $M_n$ \wrt \ $P$ can be infinite for some prime ideal $P$ of $A$.

We will need the following Lemma from \cite[1.4]{Lyu-1}.
\begin{lemma}\label{lyu-lemma}
Let $B$ be a Noetherian ring and let $L$ be a $B$-module ($L$ need not be finitely generated).
Let $P$ be a prime ideal in $B$. If $(H^j_P(L))_P$ is injective for all $j \geq 0$ then
$\mu_j(P,L) = \mu_0(P,H^j_P(L))$ for all $j \geq 0$.
\end{lemma}
The following result shows that the hypothesis of Lemma \ref{lyu-lemma} is satisfied in our case:
\begin{proposition}\label{lyu-lemma-hypoth}
(with hypotheses as in \ref{set-bass}).  Let $P$ be a prime ideal in $A$.  Let $E = \M_c$ and let $P$ be a prime ideal in $A$. Then $(H^j_P(L))_P$ is injective for all $j \geq 0$
\end{proposition}
\begin{proof}
Fix $j \geq 0$.
We note that $H^j_{PR}\circ\FF$ is a graded Lyubeznik functor on $\ ^*Mod(R)$. Also notice $H^j_{P}(E) = (H^j_{PR}(\FF(R)))_c$. Finally note that either
$H^j_P(E)_P = 0$ or $P$ is a minimal prime of $H^j_P(L)$.

We have nothing to show if $H^j_P(E)_P = 0$. So assume  $H^j_P(E)_P \neq 0$. We apply our standard technique to $H^j_{PR}\circ\FF$, see \ref{referee}.
 Notice  by \ref{min-loc} we get that
 $\N_c = E_P \neq 0$.

 As $P$ is a minimal prime of $M_c$ we get that $N_c$ is supported  ONLY at the maximal ideal of $B$. By \ref{support-m} we get that $\N_c = E_B(K)^\alpha$ where
 $E_B(K)$ is the injective hull of $K$ as a $B$-module (and $\alpha$ some ordinal possibly infinite).
 But we have
 \[
 E_B(K) \cong E_{A_P}(\kappa(P)) \cong E_A(A/P) \quad \text{as $A$-modules.}
 \]
Thus $(H^j_P(L))_P$  is an injective $A$-module.
\end{proof}

We now give
\begin{proof}[Proof of Theorem \ref{bass-basic}]
Let $P$ be a prime ideal in $A$. Fix $j \geq 0$. Suppose if possible $\mu_j(P, \M_c) < \infty$ for some $c \in \ZZ$. We show that  $\mu_j(P, \M_n) < \infty$ for all $n \in \ZZ$.

By Lemma \ref{lyu-lemma-hypoth} and Proposition \ref{lyu-lemma} we get
that $\mu_j(P, \M_n) = \mu_0(P, H^j_P(\M_n))$ for all $n \in \ZZ$.
We note that $(H^j_{PR}\circ \FF)(R)_n = H^j_P(\M_n)$ for all $n \in \ZZ$. Furthermore $\mathcal{E} = H^j_{PR}\circ\FF$ is a graded Lyubeznik functor on $\ ^*Mod(R)$. We apply our standard technique \ref{referee}.

 Notice  by \ref{min-loc} we get that
 $\N_n = (H^j_P(\M_n))_P $ for all $n \in \ZZ$.

 Note that either $(H^j_P(\M_n))_P  = 0$ OR
  $P$ is the minimal prime of $H^j_P(\M_n)$. Thus we get that $\N_n$ is supported  ONLY at the maximal ideal of $B$ for all $n \in \ZZ$. By \ref{support-m} we get that $\N_n = E_B(K)^{\alpha_n}$ where
 $E_B(K)$ is the injective hull of $K$ as a $B$-module (and $\alpha$ some ordinal possibly infinite). By \ref{referee} we may assume $K$ is infinite. We note that
 \begin{enumerate}
 \item
 $\alpha_n = \mu_j(P, \M_n)$ for all $n \in \Z$.
 \item
 $\alpha_c < \infty$.
 \end{enumerate}

    By \ref{main-koszul-body}  we get that
$V = H_g(Y_1, \ldots,Y_g;  \N)$ is $F_D$-finite where $D = K[X_1,\ldots, X_m]$ with $V_n = H_d(\bY, N)_n = K^{\alpha_n}$ for all $n \in \ZZ$.

Now $\dim_K V_c = \alpha_c < \infty$. By Proposition \ref{len-body}     we get $\dim_K V_n < \infty $ for all $n \in \ZZ$. It follows that $\mu_j(P, M_n) = \dim_K V_n < \infty$ for all $n \in \ZZ$.  Finally we note that the assertions (a), (b), (c), (d) and (e) follow from Theorems \ref{tame-body} and \ref{rigid-body}.
 \end{proof}

We now indicate
\begin{proof}[Proof of Theorem \ref{bass-m-one}]
We do same construction as in proof of \ref{bass-basic}. We use \ref{len-m-1} to conclude.
\end{proof}
Finally we give
\begin{proof}[Proof of Theorem \ref{bass-growth}]
We do same construction as in proof of \ref{bass-basic}. We use \ref{growth-body}   to conclude.
\end{proof}

\section{Associate Primes}
In this section we prove Theorem \ref{ass}.
 To prove this theorem we need a generalization of an exercise problem from
Matsumura's  classic text \cite[Exercise 6.7]{Mat}.
\begin{proposition}\label{M-ex}
Let $f \colon A \rt B$ be a homomorphism of Noetherian rings. Let $M$ be an $B$-module. Then
\[
\Ass_A M  = \{ P\cap A \mid P \in \Ass_B M \}.
\]
In particular if $\Ass_B M$ is a finite set then so is $\Ass_A M$.
\end{proposition}
\begin{remark}
Matsumura's exercise is to prove the above result for finitely generated $B$-modules. Using this fact Proposition \ref{M-ex} can be easily proved. \end{remark}

\s We note that $P \in \Ass_A V$ if and only if $\mu_0(P, V) > 0$.
 We now give
 \begin{proof}[Proof of Theorem \ref{ass}]
 By \cite[2.12]{Lyu-2} it follows that  $\Ass_R \FF(R)$ is finite.

 (1) This follows from Proposition \ref{M-ex}.

 For (2), (3) let
 \[
 \bigcup_{n \in \ZZ} \Ass_A M_n  = \{ P_1, \ldots, P_l \}.
 \]

 (2) Let $P = P_i$ for some $i$.
 Let $r \leq -m$. Then by Theorem  \ref{bass-basic} it follows that $\mu_0(P, M_r) > 0$ if and only if $\mu_0(P, M_{-m}) > 0$.   The result follows.

 (3)Let $P = P_i$ for some $i$.
 Let $s \geq 0$. Then by Theorem  \ref{bass-basic} it follows that $\mu_0(P, M_s) > 0$ if and only if $\mu_0(P, M_{0}) > 0$.   The result follows.
 \end{proof}

\section{Dimension of Support and injective dimension}
We first show
\begin{lemma}\label{injdim-dim}
(with hypotheses as in \ref{std}). Let $c \in \ZZ$. Then
\[
\injdim M_c \leq \dim M_c.
\]
\end{lemma}
\begin{proof}
Let $P$ be a prime ideal in $A$. Then by Proposition \ref{lyu-lemma-hypoth} and Lemma \ref{lyu-lemma} we get
\[
\mu_j(P, M_c) = \mu(P, H^j_P(M_c)).
\]
By Grothendieck vanishing theorem $H^j_P(M_c) = 0$ for all $j > \dim M_c$, see \cite[6.1.2]{BS}.  So $\mu_j(P, M_c) = 0$ for all $j > \dim M_c$. The result follows.
\end{proof}
We now give
\begin{proof}[Proof of Theorem \ref{injdim-and-dim}]
(1) This follows from Lemma \ref{injdim-dim}.

For (2), (3) let $P$ be a prime ideal in $A$.  Let $c \leq -m$.

(2) Fix $j \geq 0$. By
Theorem \ref{bass-basic} we get that $\mu_j(P, M_c)> 0$ if and only if
$\mu_j(P, M_{-m})> 0$. The result follows.

(3) We note that $\M_P = H(S)$ for some graded Lyubeznik functor $H$ on $S$ where $S = A_P[X_1,\ldots, X_m]$. By Theorem \ref{tame} it follows that $(M_{-m})_P \neq 0$ if and only if $(M_c)_P \neq 0$. The result follows.

(4), (5), 6(a) follow with similar arguments as in (2), (3).

For 6(b),(c)
let $P$ be a prime ideal in $A$.

6(b) By
Theorem \ref{bass-basic} we get that if $\mu_j(P, M_r)> 0$ then
$\mu_j(P, M_{-m})> 0$ and  $\mu_j(P, M_{0})> 0$. The result follows.

6(c) We note that $\M_P = H(S)$ for some graded Lyubeznik functor $H$ on $S$ where $S = A_P[X_1,\ldots, X_m]$.  By Theorem \ref{tame} it follows that $(M_{r})_P \neq 0$  implies \\ $(M_{-m})_P \neq 0$ and $(M_{0})_P \neq 0$ . The result follows.
\end{proof}

\section{Infinite generation}
In this section we prove Theorem \ref{inf-gen}. We will make the assumption $A$ is a domain to avoid trivial exceptions. For the convenience of reader we restate Theorem \ref{inf-gen}
\begin{theorem}\label{inf-gen-proof}(with hypotheses as in \ref{std}). Further assume $A$ is a domain. Assume $I \cap A \neq 0$. If $H^i_I(R)_c \neq 0$ then
$H^i_I(R)_c$ is NOT finitely generated as an $A$-module.
\end{theorem}
The proof of Theorem \ref{inf-gen-proof} is different than the proof we gave in characteristic zero.
\begin{proof}[Proof of Theorem \ref{inf-gen-proof}]
Suppose if possible $E = H^i_I(R)_c$ is finitely generated. We may localize at a prime $P$ in the support of $E$.
We have $\depth A_P = \injdim E_P \leq \dim E_P$ (here the first equality holds by \cite[3.1.17]{BH}  and the second inequality holds by \ref{injdim-dim}. But $A_P$ is regular, in particular it is Cohen-Macaulay. So $\depth A_P = \dim A_P$. So $\dim E_P = \dim A_P$. But notice $J = (I\cap A)_P  \neq 0$ and $E_P$ is $J$-torsion. In particular $\dim E_P < \dim A_P$ a contradiction.
\end{proof}

\section{Application-II}
In this section we give a proof of Theorem \ref{app-fg}

\s \label{eclair}\textbf{Setup: }Let $A = k[[Y_1, \ldots, Y_d]]$ where $k$ is an infinite field of characteristic $p > 0$ and let $R = A[X_1, \ldots, X_m]$ be standard graded. Let $I$ be a homogeneous  ideal in $R$ and let $S = R/I = \bigoplus_{n \geq 0}S_n$. Set $S_+ = \bigoplus_{n \geq 1}S_n$ be the irrelevant ideal of $S$ and let $R_+ = (X_1,\ldots, X_m)$ be the irrelevant ideal of $R$.
Let $\n$ be the unique maximal homogeneous ideal of $S$. Set
\[
\fg_\n(S) = \max \{ r \mid H^i_{\n}(S)_n = 0 \ \text{for all but finitely many $n$ and all $i < r$} \}.
\]
We restate Theorem \ref{app-fg} for the convenience of the reader,
\begin{theorem}\label{app-fg-body}
(with hypotheses as in \ref{eclair}). Let $i < \fg_n(S)$. Then \\ $H^{\dim R -i}_I(R)_n = 0$ for all $n \geq 0$.
\end{theorem}
 The following result gives a convenient criterion when $\M_n = 0$ for all $n \geq 0$ (here $\M$ is a $F_R$-finite module). Although we are only interested when $R_0 = k[[Y_1, \ldots, Y_d]]$ it is convenient to prove a more general version.
 \begin{lemma}
   \label{criterion} Let $A$ be a regular domain  of dimension $d$ containing an infinite field $k$ of characteristic $p > 0$.  Let $R = A[X_1, \ldots, X_d]$ be standard graded. Let $\M$ be a non-zero $F_R$-finite module.
   The following assertions are equivalent:
   \begin{enumerate}[\rm (a)]
     \item $\Gamma_{R_+}(\M) = \M$
     \item  $\M_n = 0$ for $n \geq 0$.
     \item  $\M_n = 0$ for $n \gg 0$.
     \item  If $P$ is an associate prime of $\M$ then $P \supseteq R_+$.
   \end{enumerate}
 \end{lemma}
 \begin{proof}
   (a) $\implies$ (b). We prove this by inducting on $d = \dim A$. When $d = 0$ we get $A$ is a field. By \cite[5.6]{MaZhang}  it follows that $\M = E(k)(m)$ where $E$ is the injective hull of $k$. In particular we have $\M_n = 0$ for all $n \geq -m + 1$. Assume the result when $\dim A < r  $ and we prove the result when $\dim A = r$.
   Suppose if possible $\M_0 \neq 0$. Let $P$ be a minimal prime of $\M_0$. Let $B = A_P$ and $R_B = B[X_1, \ldots, X_m]$. Then $ \N = \M \otimes_R R_B = \M \otimes_A B$ is a graded  $F_{R_B}$- finite module. Note $\N$ is $(R_B)_+$-torsion and $\N_0 \neq 0$. If $\dim B < r$ then we have our induction hypotheses yields a contradiction. So assume $ \dim B = r$.
   Let $C = \wh{A_P}$ be the completion of $A_P$ with respect to its maximal ideal. Note $\N_0$ has a structure as a $C$-module, see \ref{min-loc}.  Let $R_C = C[X_1, \ldots, X_m]$. Then
    $ \LL = \N \otimes_{R_B} R_C = \N \otimes_B C$ is a graded  $F_{R_C}$-finite module. Note $\LL$ is $(R_B)_+$-torsion and $\LL_0 \neq 0$. Let $C = k[[Z_1, \ldots, Z_r]]$.
    We have an exact sequence
   $$ 0 \rt H_1(Z_r, \LL) \rt \LL \xrightarrow{Z_r}  \LL \rt H_0(Z_r, \LL) \rt 0. $$
Set $\ov{C} = C/Z_rC = k[[Z_1, \ldots, Z_{r-1}]]$ and $R_{\ov{C}} = \ov{C}[X_1, \ldots, X_m] = R_C/Z_rR_C$.
 By \ref{r1} we get that $\E_i = H_i(Z_r, \LL)$ are graded $F_{R_{\ov{C}}}$-finite for $i = 0,1$. Furthermore it is easily verified that for $i = 0, 1$ the modules $\E_i$ are $(R_{\ov{C}})_+$-torsion. So by induction hypotheses $(\E_i)_n = 0$ for $n \geq 0$. In particular  we get the map on $\LL_0$ given by multiplication by $Z_r$ is an isomorphism. Therefore $\LL_0$ is a $C_{Z_r}$-module. Let $T = C_{Z_r}[X_1,\ldots, X_m]$. Then $U = \LL\otimes_{R_C}T = \LL\otimes_C C_{Z_r}$ is $F_T$-finite and $T_+$-torsion. Furthermore $U_0 = \LL_0 \neq 0$.
 But $\dim C_{Z_r} = r -1$. This contradicts our induction hypotheses. Thus $\M_0 = 0$ and so by \ref{tame} it follows that $\M_n = 0$ for $n \geq 0$.

 (b)  $\Leftrightarrow$   (c)  This follows from Theorem \ref{tame}.

 (c) $\implies$ (d). Suppose if possible there exists an associate prime $\Q$ of $\M$ with $\Q \nsupseteq R_+$. We have an exact sequence
 $$ 0 \rt \Gamma_{R_+}(\M) \rt \M \rt \N \rt 0.$$
 By \ref{mod-G} it follows that $\Q$ is an associate prime of $\N$. In particular $\N \neq 0$. Also by \ref{mod-G} we have
 $$ \Ass(\N) = \{ P \mid P \in \Ass(\M) \ \text{and} \ P \nsupseteq R_+ \}.$$
Let $V = kX_1 \oplus kX_2 \oplus \cdots \oplus kX_m$. For every associate prime $P$ of $\N$ we have $P\cap V $ is a proper $k$-subspace of $V$. As
$k$ is infinite, there exists
$$\xi \in V \setminus \bigcup_{P \in \Ass \N} P\cap V. $$
Then $\xi$ is $\N$-regular and note $\xi \in R_1$. So we have an exact sequence
$$ 0 \rt \N_{i-1} \rt \N_i  \  \quad \text{for all $i \in \Z$}.$$
In particular $\N_n \neq 0$ for all $n \gg 0$. So $\M_n \neq 0$ for all $n \gg 0$. This contradicts our assumption. Thus (d) holds.

(d) $\implies$ (a).
We have an exact sequence
 $$ 0 \rt \Gamma_{R_+}(\M) \rt \M \rt \N \rt 0.$$
 By \ref{mod-G} we have
 $$ \Ass(\N) = \{ P \mid P \in \Ass(\M) \ \text{and} \ P \nsupseteq R_+ \}.$$
 By our assumption $\Ass(\N) = \emptyset$. So $\N = 0$. It follows that $\M = \Gamma_{R_+}(\M)$.
 \end{proof}

 \s \textbf{Frobenius action on local cohomology:} (with hypotheses as in \ref{eclair}). Let $\psi \colon S \rt S$ be the Frobenius. We recall that \emph{an action of the Frobenius } on a $S$-module
 $M$ is an additive  map $f \colon M \rt M$ such that $f(s m) =s^pf(m)$ for $m \in M$ and $s \in S$. By \cite[21.14]{24} we have a Frobenius action on $H^i_\n(S)$ for all $i \geq 0$.
 It follows from the construction of Frobenius action that the Frobenius action $f \colon H^i_\n(S) \rt H^i_\n(S)$ has the property that
 $f(H^i_\n(S)_n) \subseteq  H^i_\n(S)_{np}$ for all $n \in \Z$.

 Let $M$ be a graded $*$-Artininan $S$-module with an action of Frobenius $f \colon M \rt M$ such that $f(M_n) \subseteq M_{np}$ for all $n \in \Z$. Let $M^{f^t}$ be the $S$-submodule of $M$ generated by $F^t(M)$. Then note $M^{f^t}$ is a graded $S$-submodule of $M$. We have a descending chain
 $$ M \supseteq M^f \supseteq M^{f^2} \supseteq \cdots \supseteq M^{f^t} \supseteq \cdots.$$
 As $M$ is $*$-Artinian this descending chain stabilizes. Let $M^*$ be the stable value.
 Set
 $$ \fg^*(S) =  \max \{ r \mid H^i_{\n}(S)^*_n = 0 \ \text{for all but finitely many $n$ and all $i < r$} \}.$$
 Clearly $\fg_\n(S) \leq \fg^*(S)$.

 Let $D(-)$ be the graded Matlis dual functor on $ ^*Mod(R)$.
 \begin{remark}\label{lyu-c}
   Using techniques from  Lyubeznik's paper \cite{Lyu-comp} it follows there exists a  (graded) root $E_i$ of $H^{\dim R -i}_I(R)$ such that $D(E_i) = H^i_\n(S)^* $.
 \end{remark}

 The following result implies Theorem \ref{app-fg-body}.
 \begin{theorem}\label{fg-l}
 (with hypotheses as in \ref{eclair}). Let $i < \fg^*(S)$. Then $H^i_I(R)_n = 0$ for all $n \geq 0$.
 \end{theorem}
 \begin{proof}
   Fix $i <  \fg^*(S)$. As in \ref{lyu-c}  let $E_i$ be a graded root  of $H^{\dim R -i}_I(R)$ such that $D(E_i) = H^i_\n(S)^* $. So $E_i = D(H^i_\n(S)^*)$ only has finitely many non-zero components. In particular $E_i$ is $R_+$-torsion. As $\Ass E_i = \Ass  H^{\dim R -i}_I(R)$; it follows that any associated prime $P$ of $H^{\dim R -i}_I(R)$ contains $R_+$. So by Lemma \ref{criterion} the result follows.
 \end{proof}

In view of Theorem \ref{fg-l} we study the invariant
\[
c(R, S) = \max \{ r \mid H^{\dim R - i}_{I}(R)_n = 0 \ \text{for all $n \geq  0$ and all $i < r$} \}.
\]
\s It is trivial to see that if $S_n = 0$ for $n \gg 0$ then $\fg^*(S) = + \infty$. We show
\begin{lemma}(with hypotheses as in \ref{fg-l}).
Suppose $S_n \neq 0$ for all $n \geq 0$. Then $\fg^*(S) \leq \dim S$.
\end{lemma}
\begin{proof}
Suppose if possible $\fg^*(S) > \dim S$. Then by Theorem \ref{fg-l} we get \\ $H^{\dim R - \dim S}_{I}(R)_n = 0 \ \text{for all $n \geq  0$}$. Let $g = \height I = \dim R - \dim S$.
Then it is well-known that $\Ass H^g_I(R) = $ minimal primes of $I$. By Lemma \ref{criterion} we get that $P \supseteq R_+$ for all minimal primes of $I$. It follows that $Q \supseteq R_+$ for all associate primes of $I$.
We have an exact sequence
$$ 0 \rt \Gamma_{R_+}(S) \rt S \rt \ov{S} \rt0.$$
By \ref{mod-G} it follows that $\Ass \ov{S} = \emptyset$. So $\ov{S} = 0$. Thus $S = \Gamma_{R_+}(S)$. So $S_n = 0$ for all $n \gg 0$, a contradiction.
\end{proof}

We show
\begin{proposition}\label{fg-prop}
(with hypotheses as in \ref{fg-l}). $c(R,S) = \fg^*(S)$.
\end{proposition}
\begin{proof}
We have $c(R,S) \geq c = \fg^*(S)$. If $c$ is infinite then we are done. So assume $c$ is a finite number. Suppose $c(R,S) > c$. Then $H^{\dim R - c}_{I}(R)_n = 0 \ \text{for all $n \geq  0$}.$ Let $E_c$ be the graded root of $H^{\dim R -c}_I(R)$ such that $D(E_c) = H^c_\n(S)^* $. Note as we have graded inclusion $E_c \hookrightarrow H^{\dim R -c}_I(R)$ and as $E_c$ is a finitely generated $R$-module we get that $E_c$ is concentrated in finitely many degrees.
This  implies $H^c_\n(S)^* $ is concentrated in finitely many degrees. So $\fg^*(S) \geq c + 1$, a contradiction.
\end{proof}

As an easy consequence of Theorem \ref{fg-l} we get
\begin{theorem}
  \label{fg-max} Let $(A,\m)$ be a regular local ring containing a field of characteristic $p>0$. Let $R = A[X_1, \ldots, X_m]$ be standard graded and let $I$ be a homogeneous ideal in $R$. Set $S = R/I$.  Assume also $S$ is equidimensional and $\Proj(S) \neq \emptyset$. If $\Proj(S)$ is \CM \ then $H^j_I(R)_n = 0$ for all $n \geq  0$ and all $j > \height I$.
\end{theorem}
\begin{proof}
  By \ref{proj-remark} we get $\fg_\n(S) = \dim S$. If residue field of $k$ is not finite then consider the flate extension
  $A' = A[X]_{\m A[X]}$. We then complete $A'$ and denote it by $B$. Set $R_B = R\otimes_A B = B[X_1, \ldots, X_m]$  and $I_B = IR_B$. It can be seen that $S_B = SR_B/I_B$ is a flat extension of $S$. It is easily proved that $\fg_{\n S_N} \geq \fg_\n(S) = \dim S = \dim S_B$.
  The result follows from \ref{fg-l}.
\end{proof}
\section{Application-III}
\s \label{koszul-app-body} \textbf{Setup:}
Let $(A,\m)$ be a Noetherian local ring with infinite residue field. Assume there $\pi \colon R \rt A$ is a surjection where $R$ is a  regular local ring of dimension $d$.  Let $\m$  be the maximal ideal of $R$.
We show
\begin{theorem}\label{koszul-app-th-body}
(with hypotheses as in \ref{koszul-app-body}) The Koszul cohomology modules \\ $H^j(\m, H^{d-i}_{\ker \pi}(R))$ depends only on $A,i$ and $j$ and neither  on $R$ nor on $\pi$.
\end{theorem}
To prove the result we need a few preparatory results.

\begin{lemma}\label{koz-lemma-1}
(with hypotheses as in \ref{koszul-app-body}). Let $\wh{R}$ be the completion of $R$ \wrt \ $\m$. Set $k = A/\m$ and $I = \ker \pi$. Then we have an isomorphism of $k$-vector spaces
\[
H^j(\m, H^{d-i}_{I}(R)) \cong H^j(\wh{\m}, H^{d-i}_{\wh{I}}(\wh{R})).
\]
\end{lemma}
\begin{proof}
We note that $H^j(\m, -) \cong \Ext^j_R(k, -)$ for all $j \geq 0$ as functors from $Mod(R)$ to $Mod(k)$. Therefore for all $j \geq 0$ we have,
\[
H^j(\m, H^{d-i}_{I}(R)) \cong \Ext^j_R(k, H^{d-i}_{I}(R)).
\]
We have
\begin{align*}
 H^j(\m, H^{d-i}_{I}(R)) &\cong \Ext^j_R(k, H^{d-i}_{I}(R))  \\
   &\cong \Ext^j_R(R/\m, H^{d-i}_{I}(R))\otimes_R \wh{R}, \\
   &\cong \Ext^j_{\wh{R}}(\wh{R}/\wh{\m}, H^{d-i}_{\wh{I}}(\wh{R}))\\
  &\cong H^j(\wh{\m}, H^{d-i}_{\wh{I}}(\wh{R})).
\end{align*}
Here the second isomorphism holds since $\Ext^j_R(R/\m, H^{d-i}_{I}(R))$ is a $k$-vector space.
The result follows.
\end{proof}
The following result is an essential ingredient in our proof of Theorem \ref{koszul-app-body}.
\begin{lemma}\label{quotient}
Assume $R$ is complete and $I = \ker \pi$. Furthermore assume there is a surjective map $g \colon S \rt R$ where $S$ is complete regular ring of dimension $n$. Let $J = \ker \pi\circ g$.
Then we have an isomorphism
\[
H^j(\m_R, H^{d-i}_{I}(R)) \cong  H^j(\m_S, H^{n-i}_{J}(S)).
\]
\end{lemma}
\begin{proof}
  Since $R$ is regular, $\ker g$ is generated by $n -d$ elements that form part of
a minimal system of generators of the maximal ideal $\m_S$ of $S$. By induction on $n- d$
we are reduced to the case that $n- d = 1$, so  $\ker g $ is an ideal generated by one
element $u \in \m_S \setminus \m_S^2$. By Cohen's structure theorem $S = k[[X_1,\ldots, X_{d+1}]]$ and
by a change of variables we can assume $u = X_{d+1}$.

We have an exact sequence
\[
0 \rt S \xrightarrow{X_{d + 1}} S \rt R \rt 0.
\]
We take local cohomology with respect to $J$.
We have an exact sequence for all $m \geq 0$
\[
\cdots \rt  H^m_I(R) \rt H^{m+1}_J(S) \xrightarrow{X_{d + 1}}  H^{m+1}_J(S) \rt.
\]
We note that $X_{d+1} \in J$. So $H^r_J(S)$ are $X_{d+1}$-torsion for all $r \geq 0$. By \ref{m-torsion} and \ref{m-torsion-char 0} the map
\[
H^{r}_J(S) \xrightarrow{X_{d + 1}}  H^{r}_J(S) \quad \text{is surjective for all $r \geq 0$.}
\]
Thus we have an exact sequence for all $m \geq 0$
\[
 0\rt  H^m_I(R) \rt H^{m+1}_J(S) \xrightarrow{X_{d + 1}}  H^{m+1}_J(S) \rt 0.
\]
Fix $m \geq 0$. Let $\bX^\prime = X_1, \ldots, X_d$. Taking Koszul homology \wrt \ $bX^\prime$ yields an exact sequence
\[
\cdots H_j(\bX^\prime, H^m_I(R)) \rt H_j(\bX^\prime, H^{m+1}_J(S) ) \xrightarrow{X_{d + 1}}  H_j(\bX^\prime, H^{m+1}_J(S)  ) \rt \cdots
\]
Set $\N_j = H_j(\bX^\prime, H^{m+1}_J(S) ) $ and $T = S/(\bX^\prime)$.
When characteristic $k$ is zero we have that $\N_j $ is holonomic  $\D_k(T)$-module, \cite[4.2, 4.4]{Bjork}. If characteristic $p > 0$ we have $\N_j$ is $F_T$-finite. Furthermore we note that $\N_j$ is $X_{d+1}$-torsion for all $j$.  Thus the map $\N_j \xrightarrow{X_{d+1}} \N_j$ is surjective for all $j$; see \ref{m-torsion} and \ref{m-torsion-char 0}.
It follows that  for all $j$
\begin{align*}
  H_0(X_{d+1}, \N_j)  & = 0, \\
  H_1(X_{d+1}, \N_j)  &= H_j(\bX^\prime, H^m_I(R)).
\end{align*}
Set $\bX = \bX^\prime, X_{d+1}$. It follows that  for all $j \geq 0$
\[
H_j(\bX, H^{m +1}_J(S))  \cong H_1(X_{d+1}, H_{j-1}(\bX^\prime, H^{m+1}_J(S))) \cong H_{j-1}(\bX^\prime, H^m_I(R)).
\]
It is nicer to consider Koszul cohomology.
We have
\[
H^j(\bX, H^{m+1}_J(S)) \cong  H^j(\bX^\prime, H^m_I(R)).
\]
The result follows.
\end{proof}
we now give
\begin{proof}[Proof of Theorem \ref{koszul-app-body}] Let $R \xrightarrow{\phi} A$ and $S \xrightarrow{\psi} A$ be surjections with kernels $I, J$ respectively.
By \ref{koz-lemma-1} we may assume $R, S$ are complete.  Let $T = R\wh{\otimes}_k S$ be the complete tensor product. Consider the surjection $\theta  = \phi \wh{\otimes} \psi \colon T \rt A$ and let $W = \ker \theta$.

Now $\theta$ factors through $\phi$. So we have by \ref{quotient} an isomorphism
$$H^j(\m_T, H^{\dim T - i}_W(T) ) \cong H^j(\m_R, H^{\dim R - i}_I(R) ) $$
 of $k$-vector spaces. Similarly we have an isomorphism
 $$H^j(\m_T, H^{\dim T - i}_W(T) ) \cong H^j(\m_S, H^{\dim S - i}_J(S) ) $$
  of $k$-vector spaces. The result follows.
\end{proof}

\end{document}